\documentclass[10pt]{extarticle}

\title{An efficient randomized homotopy method to approximate eigenpairs of tensors}
\author{Paul Breiding\thanks{MPI MiS Leipzig, breiding@mis.mpg.de, partially funded by DFG grant BU 1371/2-2.}}
\date{}

\usepackage[format=plain, font=footnotesize]{caption}
\usepackage[ngerman, english]{babel}
\usepackage{bibgerm}       		
\usepackage[latin1]{inputenc}		 
\usepackage{graphicx} 				
\usepackage[onehalfspacing]{setspace}
\usepackage[a4paper]{geometry}
\usepackage{url}           		
\usepackage{listings, color}		
\usepackage{subfig}
\usepackage{mathrsfs}
\usepackage{bbm}
\usepackage{scrpage2}
\usepackage{paralist}
\usepackage{tabularx}
\usepackage{tkz-euclide, tikz}
\usetikzlibrary{backgrounds,patterns,matrix,calc,arrows,positioning,fit,shapes.geometric,decorations.pathmorphing}
\usepackage{tikz-3dplot}
\usepackage{float}
\usepackage{amsthm,amsmath,amssymb,amsfonts}
\usepackage{imakeidx}
\makeindex[program=makeindex,columns=1,intoc=true, options=-s index_style.ist]
\indexsetup{toclevel=part}
\usepackage[colorlinks=true, citecolor=black, linkcolor=black, urlcolor=black, bookmarksdepth=2]{hyperref}
\usepackage{cleveref}
\usepackage{multirow}
\usepackage{enumitem}
\setenumerate{label=(\arabic*)}
\usepackage{tabularx}
\usepackage{booktabs}
\usepackage{graphicx}
\usepackage[all,2cell]{xy}
\usepackage[sort]{cite}
\usepackage{textcomp}
\usepackage{listings}
\usepackage[ruled,linesnumbered, inoutnumbered]{algorithm2e}
\SetKw{Postcondition}{Postcondition}
\usepackage[page]{appendix}

\newcommand{\HC}{\mathbb{C}}

\newcommand{\HN}{\mathbb{N}}

\newcommand{\HP}{\mathbb{P}}

\newcommand{\HR}{\mathbb{R}}
\newcommand{\HS}{\mathbb{S}}

\newcommand{\Hf}{\mathbbm{f}}

\newcommand{\cC}{\mathcal{C}}
\newcommand{\cD}{\mathcal{D}}
\newcommand{\cE}{\mathcal{E}}
\newcommand{\cF}{\mathcal{F}}
\newcommand{\cG}{\mathcal{G}}
\newcommand{\cH}{\mathcal{H}}

\newcommand{\cL}{\mathcal{L}}

\newcommand{\cO}{\mathcal{O}}
\newcommand{\cP}{\mathcal{P}}

\newcommand{\cU}{\mathcal{U}}

\newcommand{\cZ}{\mathcal{Z}}

\newcommand{\set}[1]{\left\{#1\right\}}
\newcommand{\cset}[2]{\left\{#1\mid #2\right\}}
\newcommand{\norm}[1]{\left\lvert #1 \right\rvert}
\newcommand{\Norm}[1]{\left\lVert #1 \right\rVert}
\newcommand{\rank}{\mathrm{rk}\,}
\newcommand{\id}{\mathrm{id}}
\renewcommand{\d}{\mathrm{d}}

\newcommand{\myspaceS}{ \HS(\HC^n)\times\HC}
\newcommand{\myspaceC}{ \HC^n\times\HC}
\newcommand{\myspaceCwoz}{ (\HC^n\backslash\set{0})\times\HC}

\DeclareMathOperator*{\Prob}{\mathrm{Prob}}
\DeclareMathOperator*{\mean}{\mathbb{E}}

\newcommand{\polspace}{\cH^n}
\newcommand{\hV}{{\bf V}}
\newcommand{\hW}{{\bf W}}
\newcommand{\bSigma}{\mathbf{\Sigma}}
\newcommand{\hatSigma}{\widehat{\mathbf{\Sigma}}}
\newcommand{\hatV}{\widehat{{\bf V}}}
\newcommand{\hatW}{\widehat{{\bf W}}}
\newcommand{\hatpi}{\widehat{\pi}}
\newcommand{\hatmu}{\widehat{\mu}}
\newcommand\restr[2]{\ensuremath{\left.#1\right|_{#2}}}
\newcommand{\circleradius}{0.4 cm}
\newcommand{\axislength}{0.6 cm}
\newcommand{\deriv}[2]{\mathrm{D}_{#2}#1\,}
\newcommand{\derivk}[2]{\mathrm{D}^k_{#2}#1\,}
\newcommand{\Tang}[2]{\mathrm{T}_{#1} {#2}}

\numberwithin{equation}{section}
\numberwithin{figure}{section}

\theoremstyle{plain}
\newcounter{numbering} \numberwithin{numbering}{section}
\newtheorem{thm}[numbering]{Theorem}
\newtheorem{lemma}[numbering]{Lemma}
\newtheorem{prop}[numbering]{Proposition}
\newtheorem{cor}[numbering]{Corollary}
\theoremstyle{definition}
\newtheorem{dfn}[numbering]{Definition}
\theoremstyle{remark}

\newtheorem{rem}[numbering]{Remark}
\theoremstyle{plain}
\newtheorem*{claim}{Claim}

\crefname{equation}{}{}
\crefname{equation}{}{}
\crefname{figure}{Figure}{Figures}
\crefname{section}{Section}{Sections}
\crefname{lemma}{Lemma}{Lemmata}
\crefname{prop}{Proposition}{Propositions}
\crefname{thm}{Theorem}{Theorems}
\crefname{cor}{Corollary}{Corollaries}
\crefname{dfn}{Definition}{Definitions}
\crefname{notation}{Notations}{Notations}
\crefname{rem}{Remark}{Remarks}
\crefname{claim}{Claim}{claims}
\begin{document}
\maketitle

\smallskip


\begin{abstract}
Let $A \in (\HC^{n})^{\otimes p}$ be a complex tensor of order $p$.
The pair $(v,\eta)\in\HC^n\times \HC$ is called an \emph{h-eigenpair} of $A$, if $v\neq0$ and it satisfies $Av^{p-1}=\eta^{p-2} v$, where~$Av^{p-1}$ is the contraction of $A$ by $v$ in all but the first modes. We describe a randomized algorithm to compute approximations of h-eigenpairs of complex tensors. Assuming random input, the average number of arithmetic operations it performs is polynomially bounded in the input size.
\end{abstract}
\section{Introduction}
Eigenpairs of matrices are a fundamental concept in many branches of mathematics and beyond. With the dawn of modern tensor analysis this concept has been generalized to \emph{eigenpairs of tensors} \cite{qi1,qi2,Lim2006}. Applications of eigenpairs of tensors range from mathematical disciplines like Markov Chains, Graph Theory, Optimization, Numerical Analysis and Geometry to other fields like Diffusion Tensor Imaging, Image Authenticity, Solid Mechanics or Quantum Physics. See \cite{Lim2013} for an overview. Note that in this reference several definitions of eigenpairs are mentioned, each generalizing its own particular aspect of matrix eigenpairs to higher order. In this article, however, we use the following definition: Suppose that $A=(a_{i_1,\ldots,i_{p}}) \in(\HC^{n})^{\otimes p}$ is an order $p$ tensor over the complex numbers. Then,~$A$ can be see as a \emph{multilinear map}
	\begin{equation*}
	M_A:  (\HC^n)^{\times (p-1)} \to \HC^n,\;\; (v_1,\ldots,v_{p-1})\mapsto  \begin{bmatrix} A(e_1,v_1,\ldots,v_{p-1}),&\ldots,& A(e_n,v_1,\ldots,v_{p-1})\end{bmatrix},
\end{equation*}
where $e_1,\ldots,e_n$ is the standard basis in $\HC^n$. A pair $(v,\lambda)\in(\HC^n\backslash\set{0})\times \HC$  with $\Norm{v}=1$ is called an \emph{eigenpair} of the tensor $A$, if
	\begin{equation}\label{1}
	Av^{p-1} := M_A(v,\ldots,v) = \lambda v.
\end{equation}\enlargethispage{\baselineskip}
If $(v,\lambda)$ is an eigenpair of $A$, we call $v$ \emph{eigenvector} and $\lambda$ \emph{eigenvalue} of $A$. The normalization is meaningful, because otherwise all $\lambda\in\HC$ would be eigenvalues of $A$.

In all of the aforementioned applications a main obstacle is the computational complexity of the tensor eigenpair problem \cite{Hastad1990, lim3}. One particular example of this is in the case of symmetric tensors: The \emph{spectral norm} of a symmetric tensor is a measure for entanglement of Bosons in quantum information theory; see \cite{entanglement} and \cite[Section 2.4]{Friedland2016}. For real symmetric tensors it is given by the largest real eigenvalue \cite{rank-one-approx,2011arXiv1110.5689F}, which is hard to compute. For this reason Friedland and Wang in \cite{Friedland2016} consider computing the spectral norm \emph{approximately}. They write the following.
\begin{quote}
\emph{"[Computing the spectral norm] will require solving a system of polynomial equations [...]. We conjecture [...] that [...] for most symmetric [tensors], the spectral norm~[...] is polynomial-time computable.}
\end{quote}
Motivated by this, we consider the computational complexity of numerically computing eigenpairs of general tensors (not only symmetric tensors). In this paper we show that from the numerical point of view the eigenpair problem is feasible. We provide a numerical algorithm to compute eigenpairs and this algorithm is efficient \emph{on the average}, meaning that we analyze the expected complexity of a randomized algorithm. Our main theorem is \cref{main_thm}, which informally can be stated as follows.
\begin{thm}[Informal version of \cref{main_thm}]
There is a randomized algorithm that on the average efficiently computes an eigenpair of a tensor.
\end{thm}

The basis of the present article is \cite{Armentano2015a}, in which the authors quote Demmel \cite[p. 22]{demmel}.
\begin{quote}
  \emph{
"So the problem of devising an algorithm [for the eigenvalue problem] that
is numerically stable and globally (and quickly!) convergent remains open."}
\end{quote}
The algorithm with which this problem is solved is given by Armentano in \cite{Armentano2014}. In that paper a condition number for the matrix eigenpair problem is defined and, by using this definition, in \cite{Armentano2015a} Armentano et al.\ provide a smoothed analysis for Armentano's algorithm. We extend their methods to general tensors $A\in(\HC^n)^{\otimes p}$, $p> 2$, solving Demmel's problem also for higher degrees (at least in the probabilistic understanding of 'quickly'). Our algorithm is a homotopy continuation algorithm that finds approximations of what we call \emph{h-eigenpairs} and that terminates almost surely.
The examination of complexity of this algorithm is by computing the expected value of the \emph{condition number of the eigenpair problem} (actually, of \emph{two} condition numbers; see \cref{sec:two_condition_numbers}). Thus, this work also contributes to \cite[Open Problem 14]{condition}:
\begin{quote}
  \emph{
"Provide probabilistic analyses of condition numbers for structured polynomial systems.\ [...]\ Such results would help to explain the success of numerical practice".}
\end{quote}
Previously, there have been proposed other algorithms, such as the Power Method \cite{Kolda2011}, to compute eigenpairs of tensors. However, the advantage of using homotopy continuation methods is twofold: The outputs are approximate zeros \cite{BSS}, i.e., points from where Newton's method converges quadratically fast to an actual solution, making it highly stable and, as mentioned above, we can provide a complexity analysis.

The attentative reader may ask why there is need to examine a homotopy method specifically for the eigenvalue problem. There already exists a thorough complexity analysis for homotopy continuation, see e.g., \cite[Part III]{condition} for a comprehensive study, and excellent implementations such as \cite{Bates}. The reason is that the aforementioned complexity analyses are based on the probabilistic examination of the \emph{condition number of polynomial equation solving} as defined in \cite{BSS} and with respect to this condition number the eigenpair problem is ill-posed; we will discuss this in detail in \cref{sec:1.2}. For the rest of this introductory section we now focus on explaining our approach.

It is convenient to see eigenpairs of tensors as eigenpairs of homogeneous polynomial systems as defined in \cite{distr}. To be precise, let $\cH_{n,d}$ denote the space of complex homogeneous polynomials of degree~$d$ in the~$n$ variables $X_1,\ldots,X_n$. Similar to \cref{1} we call $(v,\lambda)\in(\HC^n\backslash\set{0})\times \HC$ an eigenpair of $f\in(\cH_{n,d})^n$, if it satisfies the equation $f(v)=\lambda v$. There is a canonical surjective homomorphism $(\HC^{n})^{\otimes p}\to (\cH_{n,d})^n$ for $d=p-1$, called contraction map, that maps $A$ to~$f_A(X):=AX^d=M_A(X,\ldots,X)$, where $X=(X_1,\ldots,X_n)$. It is easy to see that
\begin{equation}\label{1.1}
Av^d=\lambda v, \text{ if and only if } f_A(v)=\lambda v.
\end{equation}
For any $s\in\HC^\times$, the pair $(v,\lambda)$ is an eigenpair of a polynomial system~$f$ if and only if $(sv,s^{d-1}\lambda)$ is an eigenpair of $f$. We call eigenpairs that are related in this way \emph{equivalent}. As already mentioned in \cite{sturmfels-cartwright} one can view equivalence classes of eigenpairs as elements in the \emph{weighted projective space} $\HP(1,\ldots,1,d-1)$. However, we lack a thorough analysis of this space and therefore avoid working with it. For this reason we aim at homogenizing the equation defining eigenpairs (although the following is not homogenizing in the classical sense we still find this naming appropriate): The \emph{punctured projective space} is defined as \begin{equation}\cP:=\HP(\HC^n\times \HC)\backslash\set{[0:1]}.\label{punctured_projective_space}\end{equation}
and we call $(v,\eta)\in\cP$ an \emph{h-eigenpair} of $f$, if $f(v)=\eta^{d-1} v$ (the name `h-eigenpair' is meant as an abbreviation of `homogeneous-equation-eigenpair'). Note that the equation defining h-eigenpairs is homogeneous of degree $d$. Using an auxiliary variable  $\ell$, we associate to $f$ the homogeneous polynomial system~$\cF_f\in(\cH_{n+1,d})^n$, where
\begin{equation}\label{dfn_F_f}
\cF_f:\myspaceC \to \HC^n, \;(X,\ell)\mapsto f(X)-\ell^{d-1}X.
\end{equation}
Clearly, $(v,\eta)$ is an h-eigenpair of $f$ if and only if $\cF_f(v,\eta)=0$. We say that $(w,\xi)\in\cP$ is an  \emph{approximate eigenpair} of $f$, if it satisfies two requirements: First, $(w,\xi)$ is an approximate zero of $\cF_f\in(\cH_{n+1,d})^n$ in the classical sense, see \cite[Section 14.1, Definition 1]{BSS}. Second, the associated zero~$(v,\eta)$ is an h-eigenpair of $f$, that is $(v,\eta)\in\cP$. If those two requirements are fullfilled, we call $(v,\eta)$ the \emph{associated eigenpair} of $(w,\xi)$. Note that the only difference from approximate eigenpairs to approximate zeros is that we exclude points in $\HP(\HC^n\times \HC)$ having the trivial solution~$[0:1]$ as associated zero. Trivial, because it yields the trivial and undesired solution $f(0)=1\cdot 0$. We conclude from \cite[Proposition~3.8]{distr} that the number of h-eigenpairs of a generic system $f$ is~$\cD(n,d):=d^n-1$.

Our model of complexity counts arithmetic operations, where taking square roots and drawing from a gaussian distribution are included. The main result is as follows. We will prove it at the end of \cref{sec:main_thm}.
\begin{thm}\label{main_thm}
There is a randomized algorithm that on input $f\in(\cH_{n,d})^n$ almost surely returns an approximate eigenpair of $f$. Its average number of arithmetic operations is $\cO(n^3+dn^\frac{5}{2}N^2)$, where $N=\dim_\HC \cH^n = n\binom{n+d-1}{n-1}$.
\end{thm}
\begin{rem}
In \cref{approx_rem} we discuss that the algorithm of \cref{main_thm} not only approximates h-eigenpairs, but also approximates eigenvectors.
\end{rem}

\begin{figure}[t]
\begin{minipage}[c]{0.5\textwidth}
\begin{center}
\begin{tikzpicture}[scale=4,cap=round,>=latex]
\draw[->] (-\axislength,0cm) -- (\axislength,0cm) node[right,fill=white] {$\zeta$};
\draw[->] (0cm,0cm) -- (0cm,0.65cm) node[above,fill=white] {$\eta$};
\draw (0cm,0cm) -- (0:\circleradius) arc (0:180:\circleradius);
\draw[fill=gray] (0cm,0cm) -- (0:\circleradius) arc (0:45:\circleradius);
\draw[fill=gray] (0cm,0cm) -- (135:\circleradius) arc (135:180:\circleradius);
\draw[gray] (\circleradius,0) -- (\circleradius,\axislength) node[above=1pt] {$\underbrace{\Norm{v}=\Norm{w}=1}$};
\draw[->,red,thick] (0cm,0cm) -- (\circleradius,0.3cm)  node[right=1pt]{$(v,\eta)$};
\draw[->,blue,thick] (0cm,0cm) -- (\circleradius,0.5cm)  node[right=1pt]{$(w,\xi)$};
\draw[] (0cm,\circleradius)  node[above left=1pt]{$[0:1]$};
\filldraw (0cm,\circleradius)  circle(0.4pt);
\end{tikzpicture}
\end{center}
  \end{minipage}\hfill
  \begin{minipage}[c]{0.5\textwidth}
    \caption{\small If $\Norm{f}=1$, all of its h-eigenpairs are in the grey area. If \textcolor{blue}{$(v,\eta)$} is an approximation of \textcolor{red}{$(v,\eta)$} then the ratio $\norm{\eta}/\Norm{w}$ is bounded, so \textcolor{blue}{$(w,\xi)$} doesn't approximate the trivial solution $[0:1]$. See also \cref{real_geometry_prop}.
    } \label{fig1}
  \end{minipage}
\end{figure}
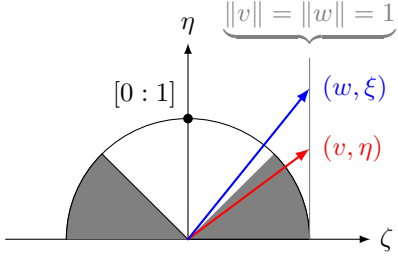

Next, we want to explain our basic ideas for constructing the algorithm that proves \cref{main_thm}. The problem of solving for h-eigenpairs consists of finding a zero of a (structured) system of $n$ homogeneous polynomials of degree~$d$ in the $n+1$ variables $X=(X_1,\ldots,X_n)$ and~$\ell$; i.e., $\cF_f(X,\ell) = f(X)-\ell^{d-1} X = 0.$
An algorithm to solve exactly those kind of polynomial systems is the \emph{adaptive linear homotopy algorithm} (ALH) introduced in \cite[Section~17]{condition}. It roughly works as follows. Let $\cG\in(\cH_{n+1,d})^n$ be a system of which a zero is known. If $\cF\in(\cH_{n+1,d})^n$ is the system to be solved, one connects $\cG$ with $\cF$ by a continuous path. This path is discretized and the zero of $\cG$ is continued along that discretized path using \emph{Newton's method}. The fineness of the discretization is determined by the \emph{condition number for polynomial equation solving}.

Algorithm ALH, however, is designed to solve general polynomial systems, not only systems of the special form, which $\cF_f(X,\ell)$ is of. It is clear that this algorithm must ignore the geometry of the h-eigenpair problem and this is what causes algorithm ALH to perform poorly on the space $\cset{\cF_f}{f\in (\cH_{n,d})^n}$; see the next subsection, where we discuss this issue in more detail. The crucial problem here is that Newton's method does not distinguish between eigenvectors and eigenvalues and simply sees eigenpairs as points in projective space. We need to make sure \emph{a priori} that we do not get too close to the trivial solution $[0:1]\in\HP^n$.

The idea is restricting the input space. Instead of having an algorithm that takes as input \emph{any} $f\in(\cH_{n,d})^n$, we only permit inputs from the unit sphere $\cset{f\in(\cH_{n,d})^n}{\Norm{f}=1}$, where~$\Norm{\cdot}$ denotes the norm induced by the Bombieri-Weyl inner product on~$(\cH_{n,d})^n$  (for details see \cref{sec:Preliminaries}). The following simple lemma is of great importance in the analysis of our analysis.
\begin{lemma}\label{useful_lemma2}
Let $f$ with $\Vert f \Vert = 1$ have the h-eigenpair $(v,\eta)$ and let $(v,\eta)\in\HC^n\times \HC$ also denote a representative. Then, $\norm{\eta}\leq \Norm{v}$.
\end{lemma}
\begin{proof}
Suppose $f=(f_1,\ldots,f_n)$. For all $i$ we have $\norm{f_i(v)}=\norm{\eta}^{d-1}\norm{v_i}$ and, by \cite[Lemma~16.5]{condition}, $\norm{f_i(v)}\leq \Norm{f_i} \Norm{v}^{d}$. Thus, $\norm{\eta}^{2(d-1)}\Norm{v}^2=\sum_{i=1}^n\norm{\eta}^{2(d-1)} \norm{v_i}^2 \leq \sum_{i=1}^n\Norm{f_i}^{2} \Norm{v}^{2d}=\Norm{v}^{2d}$,
which implies the claim.
\end{proof}

The fundamental importance of \cref{useful_lemma2} lies in implying that all h-eigenpairs of any $f$ of unit norm are sufficiently far away from the trivial solution $[0:1]$. Moreover, if we approximate carefully, all approximations of all h-eigenpairs of such an $f$ are sufficiently far away from~$[0:1]$; cf. \cref{fig1}. This restriction enables us to adapt the ideas from \cite[Section 16--17]{condition} and \cite{Armentano2015a,Armentano2014,Beltran2011,Buergisser2011,Armentano2015} to describe a homotopy method for the h-eigenpair problem.
\subsection{Why the homotopy method for general homogeneous polynomial systems is not a good choice}
\label{sec:1.2}

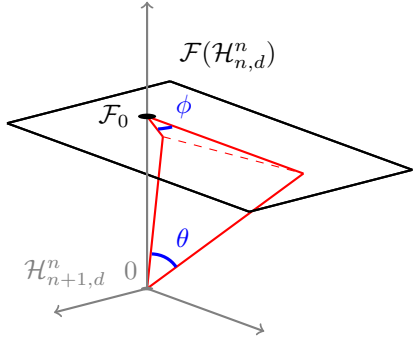
\begin{figure}[t]
\begin{minipage}[c]{0.5\textwidth}\begin{center}{
\tdplotsetmaincoords{72}{140}
\begin{tikzpicture}
		[scale = 0.8, tdplot_main_coords,
			affine/.style={thick,black},
			axis/.style={->,thick,gray}]

	\draw[-,red,thick] (0,0,0) -- (0.5,1,3);
	\draw[-,red,thick] (0,0,0) -- (0,4,3);
	\draw[-,red,thick] (0,0,3) -- (0.5,1,3);
	\draw[-,red,thick] (0,0,3) -- (0,4,3);
	\draw[dashed, red,thin] (0.5,1,3) -- (0,4,3);
	\foreach \x in {-1.5,2}
		\foreach \y in {-1.2,5}
		{
			\draw[affine] (\x,-1.2,3) -- (\x,5,3);
			\draw[affine] (-1.5,\y,3) -- (2,\y,3);
		}
	\draw[axis] (0,0,0) -- (2,0,0) node[anchor= north west]{};
	\draw[axis] (0,0,0) -- (0,3,0) node[anchor=west]{};
	\draw[axis] (0,0,0) -- (0,0,5) node[anchor=west]{};

	\fill[black,thick] (0,0,3) circle (0.15) node[anchor=east]{$\cF_0\;$};
	\fill[gray,thin] (0,0,0) circle (0.1) node[anchor=south east]{$0$};
	\fill[black,thin] (-1.5,-1.2,3) circle (0) node[anchor=south west]{$\cF(\cH_{n,d}^n)$};
	\fill[gray,thin] (2,-1,0) circle (0) node[anchor=south west]{$\cH_{n+1,d}^n$};

\tdplotdrawarc[blue,very thick]{(0,0,3)}{0.7}{65}{86}{anchor= south west}{$\phi$};
\tdplotsetthetaplanecoords{160}
\tdplotdrawarc[tdplot_rotated_coords, blue, very thick]{(0,0,0)}{0.6}{8}{58}
          {anchor=south west}{$\theta$}
\end{tikzpicture}}
\end{center}
 \end{minipage}\hfill
  \begin{minipage}[c]{0.45\textwidth}
\caption{\small The plot on the left shows a schematic picture of the affine space $\cF(\cH_{n,d}^n)$ within the ambient space $\cH_{n+1,d}^n$.
Although the angle \textcolor{blue}{$\phi$} between the two points in $\cF(\cH_{n,d}^n)$ is small the corresponding angle  \textcolor{blue}{$\theta$} with respect to $\cH_{n+1,d}^n$ can be large.
      \label{fig2}}
  \end{minipage}
\end{figure}

In this section we explain why the aforementioned algorithm ALH fails on solving the h-eigenpair problem efficiently. The complexity of algorithm ALH is measured in terms of the number of iterations it takes in $(\cH_{n+1,d})^n$ to get from one polynomial system to another. The length of each iteration step is proportional to the inverse of the \emph{condition number for polynomial equation solving}---the larger the condition number, the smaller the step size; see \cite[Algorithm~17.1]{condition}. The condition number is in fact defined with respect to the Bombieri-Weyl product on~$(\cH_{n+1,d})^n$; see \cite[(16.6)]{condition}. However, the inner product we choose for our structured set of polynomials
$\cF((\cH_{n,d})^n) = \cset{ \cF_f}{f\in(\cH_{n,d})^n} \subset (\cH_{n+1,d})^n$
is the inner product induced by the Bombieri-Weyl product on $(\cH_{n,d})^n$. Note that $\cF((\cH_{n,d})^n)$ is an affine linear subspace within~$(\cH_{n+1,d})^n$. This leads to situations where small angles in $\cF((\cH_{n,d})^n)$ may lead to large angles in the ambient space. In other words, although two systems $f$ and $g$ are close with respect to the angular measure in $\cF((\cH_{n,d})^n)$, the angular measure in $(\cH_{n+1,d})^n$ may detect them as being far from each other; c.f. \cref{fig2}. As a consequence, small perturbations in $f$ cause large perturbations in the zeros of $\cF_f$, which causes the condition number for polynomial equation solving to be large and ALH to choose an unnecessarily small step size. Nevertheless, small perturbations in $f$ do~\emph{not} cause large perturbations in the eigenvectors of $f$. This is why using the condition numbers from \cref{mu_def} for determining the step size in the homotopy methods yields an efficient algorithm for the eigenpair problem, while using the condition number for polynomial equation solving does not. In \cref{fig_experiment} we show the outcome of an experiment that compares both condition numbers.

\begin{figure}{t}
	\begin{center}
\includegraphics[scale = 0.4]{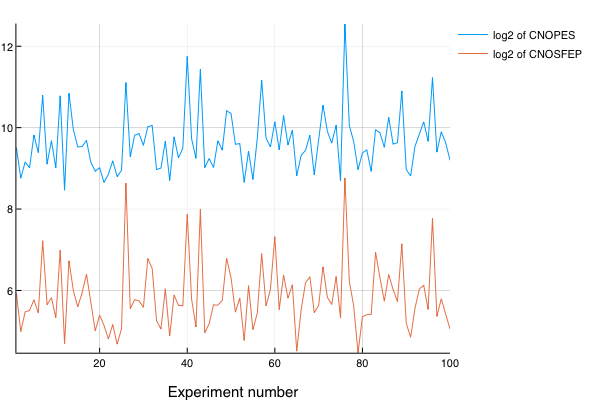}
\end{center}
\caption{The plot shows the outcome of following experiment: We sampled 100 i.i.d. samples from the uniform distribution on the unit sphere in $(\cH_{2,15})^2$. For each sample point $f$ we then computed an h-eigenpair $(v,\eta)$ using the \texttt{Julia} package \texttt{HomotopyContinuation.jl} \cite{BT}. The blue line in the plot shows the $\log_2$ of the condition number of polynomial equation solving (CNOPES) (see \cref{Proposition_1.4} for details) of each datum $(f,v,\eta)$. The red line in the plot shows the $\log_2$ of the condition number of solving for eigenpairs (CNPSFEP) as defined in \cref{mu_def} (1). The experiments are enumerated so we can compare the condition numbers for each datum. The plots were created using the package \texttt{Plots.jl}. \label{fig_experiment}}
\end{figure}

The following proposition quantifies \emph{how} ill-posed solving for eigenpairs is, when measuring its complexity in terms of the condition number for polynomial equation solving. We give a proof in \cref{sec_proof_ill_posed}. Note that the bound of $36\%$ in the proposition is obtained making very rough estimates in  the proof, so the actual probability should be even larger.
\begin{prop}\label{Proposition_1.4}
Let $f\in\cH_{n,d}^n$ and $(v,\eta)\in\cP$ be an h-eigenpair of $f$. Denote by $\mu(\cF_f,(v,\eta))$ the condition number for polynomial equation solving at $(\cF_f,(v,\eta))$ from \cite[Equation (16.6)]{condition}. If the coefficients of $f$ in the Bombieri-basis are standard normal random variables and $(v,\eta)$ is chosen uniformly at random from the $\cD(n,d)=d^n-1$ many h-eigenpairs of $f$, then
$$\Prob\set{\mu(\cF_f,(v,\eta)) > \frac{\sqrt{2}^{d-3}}{d}} > 0.36.$$
\end{prop}
In closing of this section, we remark that condition numbers are often used to detect singularities. The foregoing discussion suggests that for structured systems the respective structured condition number should be preferred for making trustworthy predictions.
\subsection{Organization}
The organization of the paper is as follows. We first establish in Section~\ref{se:geo-framework}
the geometric framework for the eigenpair problem.
Our concepts and notations are close to \cite[Section~16]{condition} and some of our results in \cite{distr} will be transferred to the h-eigenpair scenario. In Section \ref{sec:two_condition_numbers} we derive the two condition numbers of solving for h-eigenpairs using the geometric framework for condition numbers given in \cite[Section 14]{condition}. In Section  \ref{sec:main_thm} we describe the homotopy method that uses the condition number to approximate h-eigenpairs and construct from it a randomized algorithm, that proves \cref{main_thm}. In the remaining sections we give proofs.
\subsection{Acknowledgements}
The basis of this work was laid during the program "Algorithms and Complexity in Algebraic Geometry" at the Simons Institute for the Theory of Computing. We are grateful for the Simons Institute for the stimulating environment. The author was partially funded by the DFG grant BU 1371/2-2. This work is part of the author's PhD-thesis \cite{Breiding2017}. It would not have been possible without the support of Peter B\"urgisser. We are thankful for all the discussions and advice. Moreover, we want to thank Pierre Lairez for multiple discussions.
\section{Preliminaries and Notation}\label{sec:Preliminaries}
We denote h-eigenpairs with symbols $(v,\eta)$ and approximate eigenpairs with symbols $(w,\xi)$. Moreover, we will often use the same symbols for elements in $\cP$ and their representatives in~$(\HC^n\backslash\set{0})\times \HC$. On $\HC^n$ we choose the standard hermitian inner product $\langle x,y\rangle =x^T\overline{y}$. We furthermore denote the sphere in $\HC^n$ by $\HS(\HC^n):=\cset{x\in\HC^n}{\Norm{x}=1}$ and by $\HP(\HC^n)$ we denote the projective space over $\HC^n$. The real sphere is denoted by $\HS(\HR^n):=\cset{y\in\HR^n}{y^Ty=1}$. For $x,y\in\HC^n\backslash\set{0}$ we define $d_\HS(x,y)$ as the angle between the rays $\HR_{>0}x$ and $\HR_{>0}y$, where $x$ and $y$ are interpreted as elements in $\HR^{2n}$. Moreover, for $x,y\in\HP(\HC^n)$ we define the projective distance between $x$ and $y$ as  $d_\HP(x,y)$. In formulas:
$$\cos d_\HS(x,y)=  \frac{\Re(x^T\overline{y})}{\Norm{x}\Norm{y}}\quad\text{and}\quad \cos d_\HP(x,y)=\frac{\norm{x^T\overline{y}}}{\Norm{x}\Norm{y}}$$
It is well known that $d_\HS$ is the Riemannian distance on $\HS(\HC^n)$ and that $d_\HP$ is the Riemannian distance on $\HP(\HC^n)$. For a differentiable function between differentiable manifolds $f:A\to B$ we write $ \deriv{f}{x}$ for the derivative of $f$ at $x\in A$.

As in the introduction we denote by $\cH_{n,d}$ the space of homogeneous polynomials of degree~$d$ in $n$ variables.
The Bombieri-Weyl inner product on $\cH_{n,d}$ is defined as $\langle f,g\rangle = \sum_\alpha f_\alpha\overline{g_\alpha}$ for $f=\sum_\alpha f_\alpha \sqrt{\binom{d}{\alpha}}  X^\alpha$ and $g=\sum_\alpha g_\alpha \sqrt{\binom{d}{\alpha}}  X^\alpha$. For $f=(f_1,\ldots,f_n)$, $g=(g_1,\ldots,g_n)\in(\cH_{n,d})^n$ we define $\langle f,g\rangle:=\sum_{i=1}^n\langle f_i,g_i\rangle$. The Bombieri-Weyl norm is $\Norm{f}=\sqrt{\langle f,f\rangle}$.

In the following $n$ and $d$ are fixed and we abbreviate
$$\cH:= \cH_{n,d}.$$
For $x\in(\HC^n\backslash\set{0})$ let $Z(x):=\cset{f\in\cH^n}{f(x)=0}$. By \cite[Prop.~16.16]{condition} the space $\cH^n$ decomposes orthogonally as $\cH^n = C(x)\oplus L(x) \oplus R(x),$
where $R(x):=\cset{f\in\cH^n}{f(x)=0, \deriv{f}{x}=0}$ and $L(x):=R(x)^\perp \cap Z(x)$ and $C(x):=Z(x)^\perp$.

Let $W$ be a finite dimensional complex vector space with hermitian inner product and put $k:=\dim_\HC W$.
If $z\in W$ is distributed with density $\varphi_W(z) :=  \pi^{-k}\cdot \exp(-\lVert z\rVert^2),$ we say that $z$ is distributed according to the standard normal distribution on $W$ and write $z\sim N(W)$. If it is clear from the context which space is meant, we omit the subscript $W$ in $\varphi_W$.
\begin{lemma}\label{norm_inequality_gaussian}
Let $a,b\in \HC^m$ be fixed and $z\sim N_\HC(\HC^m)$. Then
\begin{enumerate}
	\item If $\Norm{a}\geq \Norm{b}$, we have
	$\Prob\set{ \Norm{z-a}\leq \Norm{z}} \leq \Prob\set{ \Norm{z-b}\leq \Norm{z}}.$
	\item For all $a$: $\Prob\limits\set{ \Norm{z-a}\leq \Norm{z}}\geq\frac{1}{\sqrt{\pi}}\,\frac{2}{\Norm{a} + \sqrt{\Norm{a}^2+8}}\, \exp\left(-\frac{\Norm{a}^2}{2}\right).$
\end{enumerate}
\end{lemma}
\begin{proof}
We have $\Norm{z-a}\leq \Norm{z}$, if and only if $\Norm{a}^2\leq 2\Re\langle z,a\rangle$. By unitary invariance we may assume that $a=(\Norm{a},0,\ldots,0)\in\HC^m$. Hence $\Re\langle z,a\rangle =\Norm{a} \Re\langle z,e_1\rangle$. Observe that $\Re\langle z,e_1\rangle$ is a real gaussian random variable with mean $0$ and variance $\frac{1}{2}$. This shows that
	\[\Prob\limits_{z\sim N_\HC(\HC^m)}\set{\Norm{z-a}\leq \Norm{z}}= \frac{1}{\sqrt{\pi}}\int_{\frac{\Norm{a}}{2}}^\infty \exp(-x^2) \d x.\]
From this equation we can deduce the first assertion. For the second assertion we use \cite[(7.1.13)]{abramowitz} on the right-hand side of the equation. This finishes the proof.
\end{proof}
For $1\leq i,j\leq m$ we denote by $A_{i,j}$ the $(i,j)$-entry of $A\in\HC^{m\times m}$ and by $A^{i,j}$ the matrix that is obtained from $A$ by removing the $i$-th row and the $j$-th column. The following is \cite[(39)]{Armentano2015a}.
\begin{lemma}\label{equation_random_matrix}
Let $A\sim N(\HC^{m\times m})$
Let $1\leq i\leq m$ and $A'$ be the matrix that is obtained by replacing the $i$-th row of $A$ by the (deterministic) $i$-th row of $\mean A$. Then, the absolute value of the expected determinant square can be expressed as $\mean \norm{\det A}^2 = \mean \norm{\det A'}^2 + \sum_{j=1}^m \mathrm{Var}(A_{i,j}) \mean\norm{\det A^{i,j}}^2.$
\end{lemma}
Recall that the Gamma function is defined by
$\Gamma(n):= \int_{t=0}^\infty t^{n-1} e^{-t} \d t$ for positive
real $n$. It is well known that $\Gamma(n+1)=n\Gamma(n)$ and,  if $n$ is a positive integer, that $\Gamma(n)=(n-1)!$.
\begin{lemma}\label{gamma_lemma} Let $n\geq 2$. We have $\Gamma(1/n)\leq \frac{\sqrt{\pi}}{2}n$ and $\Gamma(1+1/n)\leq \frac{\sqrt{\pi}}{2}$.
\end{lemma}
\begin{proof}
We have $\Gamma(1/n) = n \Gamma(1+1/n)$. The Gamma function is convex, which shows that for $n\geq 2$ we have $\Gamma(1+1/n)\leq \max\set{\Gamma(1),\Gamma(3/2)}=~\sqrt{\pi}/2$.
\end{proof}
The \emph{upper incomplete Gamma function} is defined as $\Gamma(n,x):= \int_{x}^\infty t^{n-1} e^{-t} \d t,$ where $x\geq 0$.
The following is \cite[Lemma 2.5 and Proposition 2.6]{distr}.
\begin{lemma}\label{auxiliary_lemma}
\begin{enumerate}
\item
Let $x\geq 0$ and $n\geq 1$. Then $\Gamma(n,x) = (n-1)! \, e^{-x} \sum_{k=0}^{n-1} \frac{x^k}{k!}$.
\item We have $\mean_{A\sim N(\HC^{n\times n})} \lvert \det(A)\rvert^2 = n! = \Gamma(n-1).$
\item  For $t\in \HC$ we have $\mean_{A\sim N(\HC^{n\times n})} \lvert \det (A+tI_{n})\rvert^2 =  e^{\lvert t\rvert^2} \Gamma\left(n+1,\lvert t\rvert^2\right).$
	\end{enumerate}
\end{lemma}

For $k,l\in\HN$ let $\cL_k(\HC^l,\HC^n)$ denote the vector space of multilinear maps $(\HC^l)^k\to \HC^n$. The \emph{spectral norm} of $A\in\cL_k(\HC^l,\HC^n)$ is defined as
	\begin{equation}\label{spectral_norm}
	\Norm{A}:=\sup\limits_{v_1,\ldots,v_k\in\HS(\HC^l)} \Norm{A(v_1,\ldots,v_k)}.\end{equation}
If $A\in\cL_1(\HC^l,\HC^n)$ is linear we denote by $\Norm{A}_F$ its Frobenius norm. For all $A\in\cL_1(\HC^l,\HC^n)$ one has $\Norm{A}\leq\Norm{A}_F$. The spectral norm is \emph{submultiplicative}: For $A\in \cL_k(\HC^l,\HC^n), B\in\cL_1(\HC^n,\HC^{n'})$ we have $\Norm{BA}\leq \Norm{B}\Norm{A}$. Norms of projections play an important role in this work.

\begin{lemma}\label{projection_lemma}Let $u,v\in\HC^n\backslash\set{0}$.
\begin{enumerate}
\item
 Let $P:u^\perp\to v^\perp$ be the orthogonal projection from $u^\perp$ to $v^\perp$. We have $\Norm{P}\leq 1$. If $d_\HS(u,v)<\pi/2$, then $P$ is invertible and $\Norm{P^{-1}}= [\cos  d_\HS(u,v)]^{-1}$.
\item
Let $A\in\cL_1(\HC^{n},\HC^{n-1})$, such that $\ker A =\HC v$. Then $\restr{A}{v^\perp}^{-1}\restr{A}{u^\perp}: u^\perp\to v^\perp$ is the orthogonal projection $u^\perp\to v^\perp$.
\item
Let $A\in\cL_1(\HC^{n},\HC^{n-1})$, such that $\ker A =\HC v$. Then $\Vert\restr{A}{v^\perp}^{-1}\Vert\leq \Vert\restr{A}{u^\perp}^ {-1}\Vert$.
\end{enumerate}
\end{lemma}
\begin{proof}
For (1) see \cite[Lemma 16.40]{condition}. For (2) let $x\in u^\perp$ and $y\in v^\perp$, such that $\restr{A}{v^\perp}^{-1}\restr{A}{u^\perp}x=y$. Then $Ax=Ay$, which implies that $x-y\in\ker A = \HC v$.

To prove (3), write $\restr{A}{v^\perp}^{-1} = \restr{A}{v^\perp}^{-1}\restr{A}{u^\perp}\restr{A}{u^\perp}^ {-1}.$ Using the submultiplicativity of the spectral norm we obtain $\Vert\restr{A}{v^\perp}^{-1}\Vert\leq \Vert \restr{A}{v^\perp}^{-1}\restr{A}{u^\perp}\Vert\Vert\restr{A}{u^\perp}^ {-1}\Vert.$
From (1) and (2) we get $\Vert\restr{A}{v^\perp}^{-1}\restr{A}{u^\perp}\Vert\leq 1$, which shows the assertion.
\end{proof}
\section{Geometry of the eigenpair problem}\label{se:geo-framework}
Recall from \cref{dfn_F_f} that to $f\in\polspace$ we associate the system $\cF_f(X,\ell)=f(X)-\ell^{d-1}X$, where $\ell$ is an auxiliary variable. We also recall the definition of the punctured projective space from \cref{punctured_projective_space}.
\begin{dfn}\label{def_h_eigenpairs}
$\cP:=\HP(\HC^{n}\times\HC)\backslash\set{[0:\ldots:0:1]}$ is the punctured projective space.
\end{dfn}
Moreover, we recall the definition of h-eigenpairs from the introduction
\begin{dfn}
The pair $(v,\eta)\in\HP(\HC^n \times \HC)$ is an \emph{h-eigenpair} of $f$, if $\cF(v,\eta)=0$ and $(v,\eta)\in \cP$.
\end{dfn}
In the following
for $(f,(v,\eta))\in \cH^n\times\cP$ we write $(f,v,\eta):=(f,(v,\eta))$.
An essential notion to describe the geometry of  the h-eigenpair problem is the \emph{solution variety} $\hV$ defined as
\begin{equation}\label{solution_manifold_h}
\hV:=\cset{(f,v,\eta)\in \polspace\times \cP}{ \cF_f(v,\eta)=0}.
\end{equation}\enlargethispage{\baselineskip}
Furthermore, we define the set of well-posed h-eigenpair triples
$$\hW:=\cset{(f,v,\eta)\in \hV}{ \rank \deriv{\cF_f}{(v,\eta)}=n}$$
and the set of ill-posed triples~$\bSigma\,'=\hV\backslash \hW$.

The solution variety for the eigenpair problem $V = \cset{(f,v,\lambda)\in\polspace\times\myspaceS}{f(v)=\lambda v}$ was considered in \cite{distr}. We wish to compare $V$ and $\hV$ and therefore need an appropriate set of representatives for $\hV$. A common way of representing $\cP$ as a subset of~$\HP(\HC^{n}\times\HC)$ is to choose representatives in $\HS(\HC^{n+1})$. Here, however, we choose instead $\HS(\HC^{n})\times \HC$ as the set of representatives. This reflects the distinction between eigenvector and eigenvalue and enables a comparison between $V$ and $\hV$. Note that our choice is well-defined, because $[0:\ldots:0:1]\not \in \cP$. Henceforth, we define
\begin{align*}
\hatV&:=\cset{(f,v,\eta)\in \polspace\times \myspaceS}{ \cF_f(v,\eta)=0}\quad\text{and}\\
\hatW&:=\cset{(f,v,\eta)\in \hatV}{ \rank \deriv{\cF_f}{(v,\eta)}=n}
\end{align*}
and $\widehat{\bSigma}:=\hatV\backslash\hatW$.
It is easy to see that for $t\in\HC^\times$ we have
  \begin{equation}\label{scaling_f2}(f,v,\eta)\in \hatV \text{, if and only if } (t^{d-1}f,v, t\eta)\in \hatV.\end{equation}
Pictorially, the situation appears as follows.
		\begin{equation}
		\xymatrix{
		&\ar[dl]_{\pi_1}\ar[dr]^{\pi_2}\hV& & & &\ar[dl]_{\hatpi_1}\ar[dr]^{\hatpi_2}\hatV&\\
		\cH^n&&\cP&  & \cH^n&&\HS(\HC^n)\times \HC
		}. \label{projections}
		\end{equation}
Here, $\pi_1,\pi_2,\hatpi_1,\hatpi_2$ should all denote the respective projections. We define the \emph{extended eigendiscriminant variety} $\bSigma\subset\cH_d^n$ as the Zariski closure of $\pi_1(\bSigma\,')=\hatpi_1(\hatSigma\,')$.
    \begin{rem}
    In \cite[Section 3.4]{distr} we introduced the eigendiscriminant variety $\Sigma\subset\polspace$, whose name is taken from \cite[Section 4]{Abo2015}. From \cref{deriv_psi} below one can deduce that~$\bSigma$ is the union of $\Sigma$ and all $f\in\polspace$ that have $0$ as homogeneous eigenvalue.
    \end{rem}
The sets $V$ and $\hatV$ are connected by the surjective differentiable map
\begin{equation}\label{psi_h_eigenpairs}
\psi: \hatV \to V,\; (f,v,\eta)\mapsto (f,v,\eta^{d-1}),
\end{equation}
The following lemma is straight-forward to prove.
\begin{lemma}\label{deriv_psi}
Let $(f,v,\lambda)=\psi(f,v,\eta)$.
\begin{enumerate}
  \item The derivative of $\psi$ at $(f,v,\eta)$ is given by
    $\deriv{\psi}{(f,v,\eta)}(\dot{f},\dot{v},\dot{\eta}) = (\dot{f},\dot{v},(d-1)\eta^{d-2}\dot{\eta}).$
  \item We have $\deriv{\cF_f}{(f,v,\eta)} = \deriv{F_f}{(f,v,\eta^{d-1})}\;\deriv{\psi}{(f,v,\eta)}.$
\end{enumerate}
\end{lemma}
By combining \cref{deriv_psi} with \cite[Eq. (3.3)]{distr} (or simply by derivating $\cF(X,\ell)$) we get
  \begin{equation}\label{jacobian2}
  \deriv{\cF_{f}}{(v,\eta)}= \begin{bmatrix}\deriv{f}{v}-\eta^{d-1}I_n,& -(d-1)\eta^{d-2}v\end{bmatrix}.
\end{equation}
The following is implied by \cite[Proposition 3.8]{distr}.
\begin{prop}\label{prop_eigendiscriminant}
\begin{enumerate}
\item The set $\bSigma$ is a proper subvariety of $\cH$.
\item If $f\not\in \bSigma$, $f$ has $\cD(n,d)=d^n-1$ many h-eigenpairs.
\end{enumerate}
\end{prop}
Next, we characterize the tangent space of the solution variety.
\begin{prop}\label{tangent_space}
  \begin{enumerate}
  \item The tangent space of $\hatV$ at $(f,v,\eta)\in\hatW$ equals
  	\[\cset{(\dot{f},\dot{v}+riv,\dot{\eta})\in \polspace \times (v^\perp\oplus \HR iv)\times \HC}{(\dot{v},\dot{\eta})=-\deriv{\cF_f}{(v,\eta)}|_{v^\perp\times\HC}^{-1}\dot{f}(v)}.\]
  \item The tangent space of $\hV$ at $(f,v,\eta)\in \hW$ equals
    	\[\cset{(\dot{f},\dot{v},\dot{\eta})\in \polspace \times (v,\eta)^\perp}{(\dot{v},\dot{\eta})=-\deriv{\cF_f}{(v,\eta)}|_{(v,\eta)^\perp}^{-1}\dot{f}(v)}.\]
  \end{enumerate}
\end{prop}
\begin{proof}
Let $(f,v,\lambda)=\psi(f,v,\eta)$. Then we have $\Tang{(f,v,\lambda)}{V} = \deriv{\psi}{(f,v,\eta)}\Tang{(f,v,\eta)}{\hatV}.$
Note that, unless $\eta=0$, by \cref{deriv_psi} (1) the linear map $ \deriv{\psi}{(f,v,\eta)}$ is invertible. The first assertion follows from combining \cite[Lemma 3.3]{distr} with \cref{deriv_psi}~(2). To show the second assertion let $\Pi:\hatV\to \hV$ denote the canonical projection. Recall from \cite[Section 14.2]{condition} that the tangent space $\Tang{(v,\eta)}{\cP}=\Tang{(v,\eta)}{\HP(\HC^n\times\HC)}$ can be identified with the orthogonal complement $(v,\eta)^\perp$. Hence, $\deriv{\Pi}{(f,v,\eta)}=\id_{\cH^n}\times p$, where $p:\HS(\HC^n)\times \HC \to (v,\eta)^\perp$ is the orthogonal projection. Applying this projection to the tangent space in (1) shows the tangent space in (2) has the declared form.
\end{proof}

Let $\cU(n)$ denote the group of unitary linear maps $\HC^n\to\HC^n$. We define a group action on $\polspace$ via $U.f:=U\circ f\circ U^{-1}$ and a group action on $\hV$ and $\hatV$ via \begin{equation}\label{group_action}U.(f,v,\eta):=(U.f,Uv,\eta),\end{equation} respectively. We note that $\cU(n)$ acts on $\polspace$ by isometries---see \cite[Theorem 16.3]{condition}---and that $\hW,\hatW$ both are $\cU(n)$-invariant.
Clearly, $\pi_1,\pi_2,\hatpi_1,\hatpi_2$ from \cref{projections} are all $\cU(n)$-equivariant.

As mentioned earlier in the introduction, the special structure of the polynomial system $\cF_f(X,\ell)$ requires us to bound the norm of $f$. We define
    \begin{align*}
    \hV_\HS&:=\cset{(f,v,\eta)\in \HS(\polspace)\times \cP}{ f(v)=\eta^{d-1}v},\\
    \text{and}\quad\hatV_\HS&:=\cset{(f,v,\eta)\in \HS(\polspace)\times \myspaceS}{ f(v)=\eta^{d-1}v},
    \end{align*}
and $\hW_\HS$, $\hatW_\HS$, correspondingly.
\section{Two condition numbers}\label{sec:two_condition_numbers}
The problem of solving for eigenpairs comes with two notions of condition, depending on how one defines the output. There are \emph{two} possible output spaces: $\cP$ and $\myspaceS$. Defining the output as elements in~$\cP$ respects the architecture of Newton's method (because Newton's method moves points in projective space), while defining the output as elements in $\myspaceS$ respects the geometry of the h-eigenpair problem. For each output space we will derive a condition number using the framework in \cite[Section 14]{condition}. Both definitions of condition numbers are equally important in the analysis of our algorithm and we will discuss later in this section for what purposes we will use each of them.

Recall from \cref{projections} the definitions of the projections $\pi_1,\pi_2,\hatpi_1,\hatpi_2$. If the derivative $\deriv{\pi_1}{(f,v,\eta)}$ has full rank, there exists a local solution map $S_{(f,v,\eta)} := \pi_2\circ \pi_1^{-1}$. Note that $\deriv{\pi_1}{(f,v,\eta)}$ has full rank, if and only if $(f,v,\eta)\in\hW$. Then, the \emph{condition number of solving for h-eigenpairs with output in $\cP$} at $(f,v,\eta)$ is defined as
    \begin{equation}\label{kappa1}\kappa(f,v,\eta):=
    \begin{cases}\Norm{\deriv{S_{(f,v,\eta)}}{f}},&\text{if } (f,v,\eta)\in \hW.\\ \infty,&\text{else.}\end{cases}\end{equation}
Similarly, if $(f,v,\eta)\in\hatW$, there exist a local solution map $\widehat{S}_{(f,v,\eta)} := \hatpi_2\circ \hatpi_1^{-1}$. The \emph{condition number of solving for h-eigenpairs with output in $\myspaceS$} at $(f,v,\eta)$ is defined as
    \begin{equation}\label{kappa2}\widehat{\kappa}(f,v,\eta):=
    \begin{cases}\Norm{\deriv{\widehat{S}_{(f,v,\eta)}}{f}},&\text{if } (f,v,\eta)\in \hatW.\\ \infty,&\text{else.}\end{cases}\end{equation}
The following proposition characterizes the two condition numbers.
\begin{prop}
We have
\vspace{-0.25cm}
  \begin{align*}
    \text{for } (f,v,\eta)\in\hW: \quad\kappa(f,v,\eta)
  &=  \frac{\Norm{v}^d}{\Norm{v,\eta}} \Norm{\restr{\big(\deriv{\cF_f}{(v,\eta)}}{(v,\eta)^\perp}\big)^{-1}},\\[0.1cm]
  \text{for } (f,v,\eta)\in\hatW:\quad \widehat{\kappa}(f,v,\eta)
    &=  \Norm{\big(\restr{\deriv{\cF_f}{(v,\eta)}}{v^\perp\times\HC}\big)^{-1}}.
    \end{align*}
\end{prop}
\begin{proof}
Let $(v,\eta)\in\myspaceCwoz$ denote a representative for $(v,\eta)\in\cP$. From \cref{tangent_space}~(2) we know that the tangent space of $\hW$ at $(f,v,\eta)$ equals
  	$$\cset{(\dot{f},\dot{v},\dot{\eta})\in \polspace \times (v,\eta)^\perp}{(\dot{v},\dot{\eta})=-\big(\deriv{\cF_f}{(v,\eta)}|_{(v,\eta)^\perp}\big)^{-1} \dot{f}(v)}.$$
This shows that $\deriv{S_{(f,v,\eta)}}{f}$ maps $\dot{f}$ to $-\big(\deriv{\cF_f}{(v,\eta)}|_{(v,\eta)^\perp}\big)^{-1} \dot{f}(v)$. By \cite[Eq. (14.13)]{condition} the norm on $(v,\eta)^\perp$ is $\Norm{w}_{(v,\eta)}= \frac{\Norm{w}}{\Norm{(v,\eta)}}$.
 Hence, we have
  \begin{equation*}
    \kappa(f,v,\eta)=
  \Norm{\deriv{S_{(f,v,\eta)}}{f}}
    = \Norm{v,\eta}^{-1} \max_{\lVert\dot{f}\rVert =1} \Norm{\big(\deriv{\cF_f}{(v,\eta)}|_{(v,\eta)^\perp}\big)^{-1} \dot{f}(v)}.
  \end{equation*}
By \cite[Lemma 16.6]{condition} the map $\polspace\to\HC^n$, $\dot{f}\mapsto \dot{f}(v)$ maps the unit ball in $\polspace$ onto the unit ball with radius $\Norm{v}^d$ in $\HC^n$. This yields the first assertion. For the second assertion we proceed in the exact same way, but using the identification of the tangent space from \cref{tangent_space} (1).
\end{proof}
Since our the algorithm computes \emph{approximate} solutions rather than actual solutions, we must establish a notion of condition number for all points $\cH^n\times\cP$, detached from the necessity of being points in the solution variety $\hV$. Thus, in what follows we will not work with the two condition numbers $\kappa$ and $\widehat{\kappa}$, but use the following modified condition numbers. Their definition is inspired by the relative condition number defined in \cite[(16.6)]{condition}. We also take over the notation~$\mu$ from \cite{condition} for our definition.
\begin{rem}
In the following definition for $(f,v,\eta)\in \hV$  we "almost" define $\mu(f,v,\eta)$ as $\kappa(f,v,\eta)$. The only difference is the scaling $\Norm{v}^{d-1}$ instead of $\Norm{v}^d\Norm{(v,\eta)}^{-1}$. The reason for our definition of~$\mu(f,v,\eta)$ is that it is easier to work with when proving the Lipschitz estimates in \cref{lipschitz_sec}.
\end{rem}
 \begin{dfn}\label{mu_def}
	 \begin{enumerate}
\item For $(f,v,\eta)\in \cH^n\times\cP$ we define
\begin{equation*}
\mu(f,v,\eta):=\begin{cases}
\Norm{v}^{d-1}\Norm{\big(\restr{\deriv{\cF_f}{(v,\eta)}}{(v,\eta)^\perp}\big)^{-1}},&\text{if } \restr{\deriv{\cF_f}{(v,\eta)}}{(v,\eta)^\perp} \text{ is invertible}\\
\infty, &\text{else.}
\end{cases}
\end{equation*}
\item For $(f,v,\eta)\in\cH^n\times\myspaceS$ we define
\begin{equation*}\label{mu_def2}
\hatmu(f,v,\eta):=\begin{cases}
\Norm{\big( \restr{\deriv{\cF_f}{(v,\eta)}}{v^\perp\times \HC}\big)^{-1}},&\text{if } \restr{\deriv{\cF_f}{(v,\eta)}}{v^\perp\times \HC} \text{ is invertible}\\
\infty, &\text{else.}
\end{cases}
\end{equation*}
\end{enumerate}
\end{dfn}
\begin{rem}
For all $t\in\HC^\times$ we have $\mu(f,v,\eta)=\mu(f,tv,t\eta)$. Hence, the definition of $\mu(f,v,\eta)$ is independent of the choice of a representative of $(v,\eta)\in\cP$ and is well defined on~$\cH^n \times \cP$.
\end{rem}
Note that $\mu(f,v,\eta)$ and $\hatmu(f,v,\eta)$ are \emph{absolute} condition numbers. It turns out that  under the restriction $f\in \HS(\cH^n)$ the condition number $\mu(f,v,\eta)$ can be used to describe a homotopy method for the eigenpair problem, see \cref{Theorem_4.4}. But $\mu(f,v,\eta)$ does not scale properly with~$f$, see \cref{rem2} for details (this is why we define $\mu(f,v,\eta)$ as an absolute condition number). By contrast, we can modify $\hatmu(f,v,\eta)$ as follows. If $\restr{\deriv{\cF_f}{(v,\eta)}}{v^\perp\times \HC}$ is invertible, define
\begin{equation}\label{nu_def}
\hatmu_{\mathrm{rel}}(f,v,\eta):=
\Norm{\begin{bmatrix} \Norm{f} I_{n}&0\\0&\Norm{f}^\frac{d-2}{d-1}\end{bmatrix} \big(\restr{\deriv{\cF_f}{(v,\eta)}}{v^\perp\times \HC}\big)^{-1}}
\end{equation}
and $\hatmu_{\mathrm{rel}}(f,v,\eta):=\infty$, otherwise
(here~$I_n$ denotes the $n\times n$ identity matrix). In \cref{2dmu}~(5) we show that $\hatmu_{\mathrm{rel}}(f,w,\xi)$
\emph{is} invariant under scaling of~$f$. We will make use of this property when giving a probabilistic analysis for our algorithm in \cref{mu_expectation0}. Summarizing, we need the condition number $\mu(f,v,\eta)$ to construct our algorithm and the condition number $\hatmu(f,v,\eta)$ to analyze our algorithm.

The interplay between the two condition numbers can only work properly, if they are of the overall same magnitude and if this magnitude does not differ much from the magnitudes of $\kappa(f,v,\eta)$ and $\widehat{\kappa}(f,v,\eta)$. This is why we need \cref{condition_for_eigenpairs} below. The requirement that $\Norm{f}$ is bounded is crucial in that proposition, because only then the ratio $\frac{\Norm{v}}{\Norm{\eta}}$ is bounded. We discussed this already in the introduction---\cref{fig1} and \cref{useful_lemma2} are to be highlighted here---and one can say that \cref{condition_for_eigenpairs} is the quantification of that discussion.
\begin{prop}\label{condition_for_eigenpairs} Let $f\in\HS(\cH^n)$.
  \begin{enumerate}
\item If $(f,v,\eta)\in\hW$, we have $\kappa(f,v,\eta) \leq \mu(f,v,\eta)\leq \sqrt{2}\,\kappa(f,v,\eta).$
\item If $(f,v,\eta)\in\hatW$, we have $\widehat{\kappa}(f,v,\eta) = \hatmu(f,v,\eta).$
\end{enumerate}
Let $\Pi:\myspaceS\to\cP$ denote the canonical projection.
\begin{enumerate}[resume]
\item We have $\frac{1}{\sqrt{2}}\,\hatmu(f,v,\eta)\leq \mu(f,\Pi(v,\eta)) \leq \hatmu(f,v,\eta).$
\item We have $\frac{1}{2}\,\widehat{\kappa}(f,v,\eta)\leq \kappa(f,\Pi(v,\eta)) \leq \widehat{\kappa}(f,v,\eta).$
\end{enumerate}
\end{prop}
\begin{proof}
The first claim follows from \cref{useful_lemma2}. The second claim is given by the respective definitions. It remains to prove the third assertion. Abusing notation we denote $(v,\eta):=\Pi(v,\eta)$. We have $$\mu(f,v,\eta) = \Norm{
\restr{\big(\deriv{\cF_f}{(v,\eta)}}{(v,\eta)^\perp}\big)^{-1}}\leq \Norm{
\restr{\deriv{\cF_f}{(v,\eta)}}{v^\perp \times \HC}^{-1}} = \hatmu(f,v,\eta),$$
the inequality by \cref{projection_lemma}(3). Note that we have $v^\perp \times \HC = (v,0)^\perp$. Then
	\begin{equation*}
	\hatmu(f,v,\eta) = \Norm{
\restr{\big(\deriv{\cF_f}{(v,\eta)}}{v^\perp\times \HC}\big)^{-1}}\leq \frac{1}{\cos \delta}\, \Norm{
\restr{\big(\deriv{\cF_f}{(v,\eta)}}{(v,\eta)^\perp}\big)^{-1}} = \frac{\mu(f,v,\eta)}{\cos \delta},
	\end{equation*}
by \cref{projection_lemma}, where
	$\cos \delta=\cos d_\HP((v,0),(v,\eta)) = \Norm{v}^2/(\Norm{v}\,\sqrt{\Norm{v}^2+\norm{\eta}^2}).$
We have $\Norm{v}=1$ and, since $\Norm{f}=1$, by \cref{useful_lemma2}, we also have $\norm{\eta}\leq~1$. This shows $\cos\delta \geq \frac{1}{\sqrt{2}}$. Finally, (4) is a corollary of (3).
\end{proof}
We finish this section by giving some useful properties of the condition numbers.
\begin{lemma}\label{2dmu}
\begin{enumerate}
  \item For $(f,v,\eta)\in\hatV$ we have
  $$1\leq d\,\left(\Norm{f}+\max\set{\norm{\eta}^{d-1},\norm{\eta}^{d-2}}\right)\Norm{\restr{(\deriv{\cF_{f}}{(v,\eta)}}{(v,\eta)^\perp})^{-1}}.$$
\item For $(f,v,\eta)\in\hW_\HS$ we have $1\leq 2d\mu(f,v,\eta)$.
\item If $t\to (f_t,v_t,\eta_t)$ is a curve in $\hW_\HS$, then $\Norm{\dot{v_t},\dot{\eta_t}}\leq \mu(f,v,\eta)\Norm{\dot{f_t}}$.
\item Let $U\in\cU(n)$. For all $(f,v,\eta)\in\hW$ we have $\mu(U.(f,v,\eta))=\mu(f,v,\eta)$ and for all $(f,v,\eta)\in\hatW$ we have $\hatmu(U.(f,v,\eta))=\hatmu(f,v,\eta)$
\item For all $(f,v,\eta)\in\hatV$ and all $s\in\HC^\times$ we have $\hatmu_{\mathrm{rel}}(s^{d-1}f,v,s\eta)= \hatmu_{\mathrm{rel}}(f,v,\eta).$
\end{enumerate}
\end{lemma}
\begin{proof}
Throughout the proof we choose a representative $(v,\eta)\in\myspaceS$.

If $(f,v,\eta)\not\in\hatW$, then $\restr{(\deriv{\cF_{f}}{(v,\eta)}}{(v,\eta)^\perp}$ is not invertible. Hence $\Vert{\restr{(\deriv{\cF_{f}}{(v,\eta)}}{(v,\eta)^\perp})^{-1}\Vert}=\infty$ and so the first claim is trivially true in this case. Assume now that $(f,v,\eta)\in\hatW$. From the submultiplicativity of the spectral norm we get
$$1\leq \Norm{\restr{(\deriv{\cF_{f}}{(v,\eta)}}{(v,\eta)^\perp})^{-1}} \Norm{\deriv{\cF_{f}}{(v,\eta)}}.$$
Using \cref{jacobian2} we see that
  \begin{align*}
  \Norm{\deriv{\cF_{f}}{(v,\eta)}}&=\Norm{\begin{bmatrix}\deriv{f}{v}-\eta^{d-1}I_n,& -(d-1)\eta^{d-2}v\end{bmatrix}}\nonumber
  \leq \Norm{\deriv{f}{v}} + \norm{\eta}^{d-1} +(d-1)\norm{\eta}^{d-2};
\end{align*}
  for the last line we have used the triangle inequality. From \cite[Lemma 16.46]{condition} we get the inequality $\Norm{\deriv{f}{v}}\leq d\Norm{f}$. Using that $d\geq 1$ shows the first assertion.

 The second assertion follows from the first by using \cref{useful_lemma2}: $\norm{\eta}\leq \Norm{v}=1$ for $\Norm{f}=1$.

For (3) let $S$ be the solution map from the beginning of \cref{sec:two_condition_numbers}. We take derivatives on both sides of $S_{(f_t,v_t,\eta_t)}(f_t)=(v_t,\eta_t)$ to get $\deriv{S_{(f_t,v_t,\eta_t)}}{f_t}(\dot{f_t}) = (\dot{v_t},\dot{\eta_t}),$ such that
	$$\Norm{\dot{v_t},\dot{\eta_t}} = \Norm{\deriv{S_{(f_t,v_t,\eta_t)}}{f_t}(\dot{f_t})} \leq  \Norm{\deriv{S_{(f_t,v_t,\eta_t)}}{f_t}}\; \Norm{\dot{f_t}}=\kappa(f_t,v_t,\eta_t)\; \Norm{\dot{f_t}}.$$
Using \cref{condition_for_eigenpairs} (1) shows then the third assertion.

A straightforward calculations shows the fourth assertion so that it remains to prove (5). First, observe that $\hatmu_{\mathrm{rel}}(f,v,\eta)<\infty$ if and only if $\hatmu_{\mathrm{rel}}(s^{d-1}f,v,s\eta)<\infty$. From \cref{jacobian2} we have
\begin{align}\nonumber
\deriv{\cF_{s^{d-1}f}}{(v,s\eta)}
&= \begin{bmatrix}s^{d-1} \deriv{f}{v}-s^{d-1}\eta^{d-1}I_n,& -(d-1)s^{d-2}\eta^{d-2}w\end{bmatrix}\\
&=\deriv{\cF_{f}}{(v,\eta)}\,\begin{bmatrix}
s^{d-1}I_{n}&0\\
0&s^{d-2}
\end{bmatrix}.\label{hallo123}
\end{align}
Using that
\begin{equation*}\label{eq_muhat}
\deriv{\cF_{f}}{(v,\eta)} \restr{\begin{bmatrix}
s^{d-1}I_{n}&0\\
0&s^{d-2}
\end{bmatrix}}{v^\perp\times\HC} =\restr{\deriv{\cF_{f}}{(v,\eta)}}{v^\perp\times\HC}
\restr{\begin{bmatrix}
s^{d-1}I_{n}&0\\
0&s^{d-2}
\end{bmatrix}}{v^\perp\times\HC}
\end{equation*}
we obtain with \cref{hallo123}
 \begin{align*}
\hatmu_{\mathrm{rel}}(s^{d-1}f,v,s\eta)=&\Norm{\begin{bmatrix}
\Norm{s}^{d-1}\Norm{f}I_{n}&0\\
0&\Norm{s}^{d-2}\Norm{f}^\frac{d-2}{d-1}
\end{bmatrix}
\restr{\deriv{\cF_{s^{d-1}f}}{(ws,\eta)}}{v^\perp\times\HC}^{-1}}\\
=&\Norm{\begin{bmatrix}
\Norm{f}I_{n}&0\\
0&\Norm{f}^\frac{d-2}{d-1}
\end{bmatrix}
\restr{\deriv{\cF_{f}}{(v,\eta)}}{v^\perp\times\HC}^{-1}}\;
= \;\hatmu_{\mathrm{rel}}(f,v,\eta).
\end{align*}
This finishes the proof.
\end{proof}
\begin{rem}\label{rem2}
The proof above shows why the condition number $\mu(f,v,\eta)$ does not scale properly with $f$. In fact, we cannot adapt equation \cref{eq_muhat} to $\mu(f,v,\eta)$, because in general we have
\begin{equation*}
\deriv{\cF_{f}}{(v,\eta)} \restr{\begin{bmatrix}
s^{d-1}I_{n}&0\\
0&s^{d-2}
\end{bmatrix}}{(v,\eta)^\perp} \neq \restr{\deriv{\cF_{f}}{(v,\eta)}}{(v,\eta)^\perp}
\restr{\begin{bmatrix}
s^{d-1}I_{n}&0\\
0&s^{d-2}
\end{bmatrix}}{(v,\eta)^\perp}.
\end{equation*}
\end{rem}
\section{The adaptive linear homotopy method}\label{sec:main_thm}
We have now discussed the geometry of the h-eigenpair problem and introduced the two condition numbers. It is time to define our algorithm.
\subsection{Newton method and approximate eigenpairs}
Let $f\in\polspace$. The Newton operator for the h-eigenpair problem with respect to $f$ is the usual Newton operator associated to the polynomial system $\cF_f:\HC^n\times\HC\to \HC^n$, that is
	\begin{align*}
	&\mathrm{Newton}_f: \quad\cP \to \cP\\
	&(v,\eta)\mapsto\begin{cases}(v,\eta)-\big(\deriv{\cF_f}{(v,\eta)}|_{(v,\eta)^\perp}\big)^{-1} \cF_f(v,\eta),&\text{if } \deriv{\cF_f}{(v,\eta)}|_{(v,\eta)^\perp} \text{ invertible} \\ (v,\eta),&\text{else}.\end{cases}
	\end{align*}
The operator $\mathrm{Newton}_f$ is a well defined rational map and its image is contained in $\cP$. The following definition was already made in the introduction. It is \cite[Section 14.1, Definition~1]{BSS} specialized to our scenario.
\begin{dfn}\label{def_approx_eigenpair}
Let $(f,v,\eta)\in\hV$. We say that $(w,\xi)\in\cP$ is an \emph{approximate eigenpair} of $f$ with associated h-eigenpair $(v,\eta)$, if for all $i\geq 0$, we have
 $$\forall i\geq 0: \; d_\HP((w_i,\xi_i),(v,\eta))\leq \frac{1}{2^{2^i-1}} d_\HP((w,\xi),(v,\eta)),$$
where $(w_0,\xi_0)=(w,\xi)$ and $(w_i,\xi_i)=\mathrm{Newton}_f(w_{i-1},\xi_{i-1}), i\geq 1$.
\end{dfn}
For the following theorem we need the functions $\delta(r)$ and $u(r)$, which are used to indicate the size of basins of attraction of the Newton method. They are defined as in \cite[p. 316]{condition}:
\begin{equation}\label{u_and_delta}
\Psi_\delta(u):=(1+\cos\delta)(1-u)^2-1,\;\delta(r):=\min\limits_{\delta>0} \set{\sin\delta =r \delta},\; u(r):=\min\limits_{u>0}\set{ 2u=r\Psi_{\delta(r)}(u)},
\end{equation}
for $u,\delta\in\HR$. The following definition is \cite[Definition 16.35]{condition}, the only difference being the notation $\gamma(f,v,\eta)$, which in that reference would be denoted $\gamma(\cF_f,(v,\eta))$:
\begin{equation}\label{gamma_def}\gamma(f,v,\eta):=\Norm{v,\eta}\;\max\limits_{k\geq 2} \Norm{\frac{1}{k!}\; \big(\deriv{\cF_f}{(v,\eta)}|_{(v,\eta)^\perp}\big)^{-1} \; \derivk{\cF_f}{(v,\eta)}}^\frac{1}{k-1}\end{equation}
(the norm is the norm for multilinear operators as defined in \cref{spectral_norm}).
The following theorem is \cite[Theorem 16.38]{condition} specialized to our scenario.
\begin{thm}\label{gamma_thm}
Let $(f,v,\eta)\in \hW$ and $(w,\xi)\in\cP$ be such that there exists some $\frac{2}{\pi} \leq r < 1$ with $d_\HP((v,\eta),(w,\xi)) \leq \delta(r)$ and $d_\HP((v,\eta),(w,\xi)) \gamma(f,v,\eta)\leq u(r),$ where $\delta(r)$ and $u(r)$ are as in~\cref{u_and_delta}. Then $(w,\xi)$ is an approximate eigenpair of $(v,\eta)$.
\end{thm}
It turns out that, if $\Norm{f}=1$, then $\gamma(f,v,\eta)$ can be bounded from above by the condition number, that we introduced in \cref{mu_def}. The following is a version of the higher derivative estimate from \cite[Theorem 16.1]{condition}. Its proof is postponed to \cref{condition_number_thm_proof}.
\begin{thm}\label{Theorem_4.3}
For $(f,v,\eta)\in\hV_\HS$ we have $\gamma(f,v,\eta)\leq \mu(f,v,\eta) d^2 \sqrt{2n} .$
\end{thm}
At this point it should not be surprising anymore that the requirement $\Norm{f}=1$ is crucial.

\subsection{The algorithm EALH}
We can now state the algorithm EALH (eigenpair adaptive linear homotopy). This algorithm lays the basis for two other subsequent algorithms that are randomized versions of EALH. Because EALH needs an initial solution $(g,v,\eta)$ as input and we want to get rid of this requirement, we will later choose the initial solution randomly. Let
\begin{equation}\label{Theta}\Theta(\varepsilon):= 1-(1-\varepsilon)^{-2}+\cos(\tfrac{\varepsilon}{4})-\varepsilon.\end{equation}
The algorithm EALH is given as follows.

\medskip
\begin{algorithm}[H]
  \caption{Adaptive linear homotopy method for eigenpairs (EALH)\label{ALH}}
\SetAlgoLined
\KwIn{$f\in \HS(\cH^n)$ and $(g,v,\eta)\in \hV_\HS$, $g\neq \pm f$.}
\KwOut{An approximate eigenpair $(w,\xi)\in \cP$ of $f$.}
Initialize: $\tau\leftarrow 0$\;
Initialize: $q\leftarrow g$, $(w,\xi)\leftarrow(v,\eta)$\;
Initialize: $\varepsilon \leftarrow 0.04$, $\chi\leftarrow 2\Theta(\varepsilon)-1$, $\alpha\leftarrow d_\HS(f,g)$\;
\While{$\tau <1$}{
$\Delta\, \tau \leftarrow \chi \,\varepsilon (1-\varepsilon)^4 \Theta(\varepsilon) / (4\alpha d^{2} \sqrt{n}\, \mu(q,w,\xi)^2)$\;
$\tau\leftarrow\min\set{1,\tau+\Delta\,\tau}$\;
$q\leftarrow$ point on the geodesic path in $\HS(\cH^n)$ from $f$ to $g$ with $d_\HS(q,g)\leftarrow\tau\alpha$\;
$(w,\xi) \leftarrow \mathrm{Newton}_q(w,\xi)$\;
}
\Postcondition The algorithm halts, if $E_{f,g}$ does not intersect $\bSigma$\;
\end{algorithm}
\medskip
\begin{rem}\label{postconditions}
\cref{prop_eigendiscriminant} tells us that $\bSigma$ is of real codimension at least 2. This shows that for almost all $f$ and $(g,v,\eta)$ the postconditions of are fulfilled.
\end{rem}
We denote by $K(f,(g,v,\eta))$ the number of iterations that EALH performs on inputs $f\in\polspace$, $(g,v,\eta)\in \hV_\HS$. The proof of the following analysis is postponed to \cref{ALH-proof}.
\begin{thm}[Analysis of EALH]\label{Theorem_4.4}
If $E_{f,g}\cap\bSigma=\emptyset$, then algorithm EALH terminates, the output $(w,\xi)$ is an approximate eigenpair of $f$ and we have
$$K(f,(g,v,\eta))\leq 246\;d^2\sqrt{n}\; d_\HS(f,g)\;\int_0^1\mu(q_\tau,v_\tau,\eta_\tau)^2 \d \tau,$$
where $\cset{(q_\tau,v_\tau,\eta_\tau)}{0\leq \tau\leq 1}$ denotes the lifting of $E_{f,g}$ at $(g,v,\eta)$.
\end{thm}
In the following we denote
\begin{equation}\label{notation_C}\cC(f,(g,v,\eta)) :=\int_0^1\mu(q_\tau,v_\tau,\eta_\tau)^2 \d \tau,\end{equation}
where $(q_\tau,v_\tau,\eta_\tau)$ is defined as in \cref{Theorem_4.4}.

\begin{rem}\label{approx_rem}
Algorithm EALH not only finds approximations of h-eigenpairs, but also approximations of eigenvectors. To see this, let $f\in\HS(\cH^n)$ be the input of algorithm EALH and suppose that we have found $(w,\xi)\in \HP(\HC^n\times\HC)$ with $$\delta:=d_\HP((v,\eta),(w,\xi)) < \frac{\pi}{16}$$
for some h-eigenpair $(v,\eta)$ of $f$. Let $(v,\eta),(w,\xi)\in\myspaceS$ also denote representatives with $\delta= d_\HS((v,\eta),(w,\xi))$.  By \cref{useful_lemma2} we have $\norm{\eta}\leq 1$. Using this and \cref{real_geometry_prop} we may deduce the inequalities $\delta\geq \sqrt{8}^{\,-1}\Norm{v-w,\eta-\xi} \geq\sqrt{8}^{\,-1} \Norm{v-w}$.
Since $v,w$ are points in the sphere $\HS(\HC^n)$, we have $\Norm{v-w}\geq \sin d_\HS(v,w)\geq \sin d_\HP(v,w)$. Using that for all $0\leq \phi\leq \frac{\pi}{2}$ we have $\sin\phi \geq \frac{2}{\pi}\phi$, we finally obtain
	$$d_\HP((v,\eta),(w,\xi))=\delta \geq \frac{1}{\pi\sqrt{2}}\; d_\HP(v,w).$$
In other words, if $(w,\xi)$ approximates $(v,\eta)$, then $w$ approximates $v$.
\end{rem}
\subsection{The algorithm LVEALH}\label{sec:LVEALH}
We now want to modify the algorithm EALH, such that the initial system together with one of its h-eigenpairs is chosen randomly. For this we need to define a probability distribution on $\hV_\HS$.
\begin{enumerate}
\item Choose $g\in\HS(\cH^n)$ uniformly at random.
\item Choose $(v,\eta)\in \cP$ uniformly from the $\cD(n,d)$ many eigenpairs of $g$.
\end{enumerate}
Compare \cite{distr}, where we proceeded similarly.
We denote the distribution that this procedure implies on the space of h-eigenpairs by $\rho$. The algorithm Las Vegas EALH (LVEALH) is as follows.

\medskip
\begin{algorithm}[H]
  \caption{Las Vegas EALH (LVEALH)}
\SetAlgoLined
\KwIn{$f\in \HS(\HC^n)$.}
\KwOut{An approximate eigenpair $(w,\xi)\in \cP$ of $f$.}
Choose $(g,v,\eta)\in\hV$ with density $\rho$\;
Start algorithm EALH with inputs $(g,v,\eta)$ and $f$.\;
Set $(w,\xi) \leftarrow$ output of EALH.\;
\Postcondition The algorithm halts, if $E_{f,g}$ does not intersect $\bSigma$\;
\end{algorithm}
\medskip
%
%
%
A little care should be taken here. We have not given an algorithm to draw from $\rho$. And, unfortunately, we are not able to provide such an algorithm. The reason why we still describe LVEALH is that we can provide an average analysis for it and this will be relevant for the analysis of the subsequent true randomized algorithm LVEALHWS; see \cref{Theorem_4.8}.

The proof of the following theorem is postponed to \cref{sec:avg_anal}.
\begin{thm}[Average analysis of LVEALH]\label{Theorem_4.6}
We have
$$\mean\limits_{f \sim\mathrm{Unif}\;\HS(\cH^n)}\; \mean\limits_{(g,v,\eta)\sim\rho} \; \cC(f,(g,v,\eta)) <   \frac{40\pi\,nN}{d}$$
where $\cC$ is as in  \cref{notation_C} and where $N=\dim_\HC(\polspace)$.
\end{thm}
\subsection{The algorithm LVEALHWS}
We now describe an algorithm that defines another probability density function on~$\hV_\HS$, which we denote by $\rho^\star$. The inspiration for this algorithm is \cite[Section 10.1]{Armentano2015a}. It roughly works as follows. Beltr\'an and Pardo \cite{Beltran2011} gave an algorithm to sample a system $\Hf$ of $n-1$ polynomials of degree $d$ in $n$ variables and a zero uniform at random from the $d^{n-1}$ many zeros of $\Hf$. I.e.; Beltr\'an-Pardo randomization yields  $\Hf\sim N(\cH^{n-1})$ and $v\in\HP(\HC^n)$ with $\Hf(v)=0$. To get from this a polynomial system together with an h-eigenpair we sample one further single polynomial $\Hf_0\in\cH$ and some~$\eta$ with $\Hf_0(v)=\eta^{d-1}$, so that, by construction, we get
$$\begin{bmatrix} \Hf_0(v)\\\Hf(v)\end{bmatrix} = \begin{bmatrix} \eta^{d-1}\\0\end{bmatrix} = \eta^{d-1} e_1.$$
If we sample~$U\in\cU(n)$ under the constraint that $Ue_1=v$, where $e_1:=(1,0,\ldots,0)$ and put $f:=U[\Hf_0, \Hf]$, we have, as desired,
$$f(v):=U\begin{bmatrix}\Hf_0(v)\\\Hf(v)\end{bmatrix} =\eta^{d-1} Ue_1=\eta^{d-1}v.$$
The discussion in \cite{distr} made us realize how one must sample the eigenvalue $\eta$: Draw $r>0$ with density $e^{-r}$, draw $\phi\in [0,2\pi)$ uniform at random and then put $\eta:=r^{\frac{1}{2(d-1)}} \exp(i\phi)$. We give a precise description of the density of $\eta$ in \cref{beta_density} below. The explicit algorithm is as follows.

\medskip
\begin{algorithm}[H]
  \caption[Sampling method for the starting system for EALH]{Draw-from-$\rho^\star$.}
\SetAlgoLined
\KwIn{-.}
\KwOut{$(f,v,\eta)\in\hV$.}
Use Beltr\'an-Pardo to draw $\Hf\sim N(\cH^{n-1})$ and $v\in\HP(\HC^n)$ from the $d^{n-1}$ many zeros of $\Hf$\;
Draw $a \sim N(v^\perp)$ and $h\sim N(R(v))$, where $R(v)=\cset{h\in\polspace}{h(v)=0, \deriv{h}{v}=0}$ is as in \cref{se:geo-framework}\;
Draw $U\in\cU(n)$, such that $Ue_1=v$, uniformly at random\;
Draw $r$ with density $e^{-r}\;\mathbf{1}_{\set{r\geq 0}}(r)$, choose $\phi\in[0,2\pi)$ uniformly, put $\eta:=r^\frac{1}{2(d-1)} \exp(i\phi)$\;
Define $\restr{I}{v^\perp}$ as the restriction of the identity to $v^\perp$\;
\uIf{$\Norm{\restr{\deriv{\Hf}{v}}{v^\perp}+\eta^{d-1}\restr{I}{v^\perp}}_F \leq \Norm{\restr{\deriv{\Hf}{v}}{v^\perp}}_F$\;}
{Put $\Hf_0:=\eta^{d-1} \;\langle X,v\rangle^d + \sqrt{d}\;\langle X,v\rangle^{d-1} a^T X + h$\;}
\Else{Delete $\eta,\phi$ and $r$ and go to 3\;}
Put $f:=U \begin{bmatrix}  \Hf_0\\ \Hf\end{bmatrix}$\;
{\bf Return} $(f,v,\eta)\in \hV$\;
\end{algorithm}
\medskip

The advantage of $\rho^\star$ when compared to $\rho$ is that we have the next proposition.
\begin{prop}\label{Theorem_4.7}
We assume to have a random number generator that can draw von $N(\HC)$ in~$\cO(1)$. We count draws as arithmetic operations. The algorithm Draw-from-$\rho^\star$ can be implemented so its expected number of arithmetic operations is $\cO(n^3+dnN)$, where $N=\dim_\HC(\polspace)$.
\end{prop}
The proof of the proposition is postponed to \cref{sec:Theorem_4.7}. Generating a starting system for the algorithm EALH with Draw-from-$\rho^\star$ yields the algorithm we call LVEALH with sampling.

\medskip
\begin{algorithm}[H]
  \caption{Las Vegas EALH with sampling (LVEALHWS).}
\SetAlgoLined
\KwIn{$f\in \HS(\HC^n)$.}
\KwOut{An approximate eigenpair $(w,\xi)\in\cP$ of $f$.}
Run Draw-from-$\rho^\star$\;
Set $(g,v,\eta)\leftarrow $ output of Draw-from-$\rho^\star$\;
Put $\tilde g=g/\Norm{g}$ and $\tilde \eta= \eta/\Norm{g}^\frac{1}{d-1}$\;
Start algorithm EALH with inputs $(\tilde g,v,\tilde \eta)$ and $f$\;
Set $(w,\xi) \leftarrow$ output of EALH\;
\Postcondition The algorithm halts, if $E_{f,g}$ does not intersect $\bSigma$.
\end{algorithm}
\medskip

We can use the average analysis of algorithm LVEALH given in \cref{Theorem_4.6} to analyze LVEALHWS. The proof of the next proposition is postponed to \cref{sec:rho_star}.
\begin{prop}\label{Theorem_4.8}
We have
\begin{align*}
&\mean\limits_{f \sim \mathrm{Unif}\,\HS(\polspace)} \mean\limits_{(g,v,\eta)\sim \rho^\star}
\cC (\tilde g,v,\tilde \eta)
\leq 10\sqrt{\pi}\,n\;\mean\limits_{f\sim \mathrm{Unif}\,\HS(\polspace)}\mean\limits_{(g,v,\eta)\sim \rho} \cC(f,(g,v,\eta)),
\end{align*}
where $(\tilde g, v, \tilde \eta)$ is the scaled h-eigenpair given by $\tilde g=g/\Norm{g}$ and $\tilde \eta= \eta/\Norm{g}^\frac{1}{d-1}$.
\end{prop}
Now we have gathered all the ingredients to prove \cref{main_thm}.
\begin{proof}[Proof \cref{main_thm}]
We prove that algorithm LVEALHWS has the properties stated. First, note that from \cref{postconditions} we get that algorithm LVEALHWS terminates almost surely. \cref{Theorem_4.7} states that expected number of arithmetic operations of algorithm Draw-from-$\rho^\star$ is~$\cO(n^3+dnN)$. The cost of every iteration in algorithm EALH is dominated by the costs of evaluating $\mathrm{Newton}_f(v,\eta)$, which by~\cite[Proposition 16.32]{condition} is $\cO(N)$. From \cref{Theorem_4.4} we get that the expected number of iterations algorithm EALH is at most
	\begin{align*}
	\mean\limits_{f \sim \mathrm{Unif}\,\HS(\polspace)}\mean\limits_{(g,v,\eta)\sim \rho^\star} K(\tilde g,v,\tilde \eta)
	\leq &246\;d^2\sqrt{n}\; d_\HS(f,g)\, \mean\limits_{f \sim \mathrm{Unif}\,\HS(\polspace)}\mean\limits_{(g,v,\eta)\sim \rho^\star} \cC (\tilde g,v,\tilde \eta
	\end{align*}
Combining \cref{Theorem_4.8} with \cref{Theorem_4.6} we see that
	$$\mean\limits_{f \sim \mathrm{Unif}\,\HS(\polspace)}\mean\limits_{(g,v,\eta)\sim \rho^\star} \cC (\tilde g,v,\tilde \eta) \leq  \frac{400\sqrt{\pi}^3\,n^2N}{d}, \quad\text{ where } \tilde g=\frac{g}{\Norm{g}}, \;\tilde \eta=\frac{\eta}{\Norm{g}^\frac{1}{d-1}}$$
and hence, $$\mean\limits_{f \sim \mathrm{Unif}\,\HS(\polspace)}\mean\limits_{(g,v,\eta)\sim \rho^\star} K(\tilde g,v,\tilde \eta)\leq 98400 d \sqrt{\pi}^3\,n^\frac{5}{2}\,N\,d_\HS(f,g).$$
Using that $d_\HS(f,g)\leq \frac{\pi}{2}$ we see that the expected number of iterations of EALH is $\cO(dn^\frac{5}{2}N)$ and so the expected number of arithmetic operations of EALH is $\cO(dn^\frac{5}{2}N^2)$. Altogether we have shown that the expected number of arithmetic operations of LVEALHWS is $\cO(n^3+dn^\frac{5}{2}N^2)$.
\end{proof}
\section{Auxiliary results}
In this section we give a number of auxiliary results that are needed to prove the theorems and proposition from the preceding section.
\subsection{Integration on the solution manifold}
Recall from \cref{projections} the projections $\hatpi_1:\hV\to \cH^n $ and $\hatpi_2:\hV\to \myspaceS$. We define the \emph{average} of a measurable map $\Xi:\hatV\to \HR$ as
\begin{equation}\label{part2.2:f_av_dfn2}
\widehat{\Xi}_\mathrm{\;av}(f):=\frac{1}{2\pi\cD(n,d)}\; \int_{(f,v,\eta)\in \hatpi^{-1}(f)} \widehat{\Xi}(f,v,\eta) \;\d(f,v,\eta).
\end{equation}
The following technical result will be needed in various occasions throughout the paper.
\begin{lemma}\label{part2.2:integration}
Suppose that $\widehat{\Xi}:\hatV\to\HR$ is a measurable and unitarily invariant function; that is, for all $U\in\cU(n)$ we have $\widehat{\Xi}(U.(f,v,\eta))=\widehat{\Xi}(f,v,\eta).$ Let $I_{n-1}$ denote the $(n-1)\times(n-1)$ identity matrix. Then,
\begin{equation*}
\mean\limits_{f\sim N(\polspace)}  \widehat{\Xi}_{\mathrm{av}}(f)=\frac{\mathrm{vol}\;\HS(\HC^n)}{2\pi^{n+1}\cD(n,d)} \int_{\eta\in\HC} (d-1)^2 \widehat{E}(\eta) \norm{\eta}^{2(d-2)} e^{-\norm{\eta}^{2(d-1)}} \,\d \eta
\end{equation*}
where
	\[\widehat{E}(\eta):= \mean\limits_{A,a,h} \widehat{\Xi}(f,e_1,\eta) \norm{\det(\sqrt{d}A-\eta^{d-1}I_{n-1})}^2,\]
  and where $A\sim N(\HC^{(n-1)\times(n-1)})$, $a\sim N(\HC^{n-1})$, $h\sim N(R(e_1))$ are independent, and where $R(e_1) = \cset{h\in\polspace}{h(e_1)=0, \deriv{h}{e_1}=0}$ and
  $$
  f= \eta^{d-1}X_1^d e_1 + \sqrt{d}X_1^{d-1}\begin{bmatrix} a^T\\ A\end{bmatrix} (X_2,\ldots,X_n)^T + h.$$
\end{lemma}
\begin{proof}
  Recall from \cref{psi_h_eigenpairs} the map $\psi: \hatV \to V,\; (f,v,\eta)\mapsto (f,v,\eta^{d-1})$. Using \cite[Lemma 3.10]{distr} and making a change of variables through $\psi$ yields the assertion.
\end{proof}
Similar to \cref{part2.2:f_av_dfn2} for a measurable map $\widehat{\Xi}: \hV \to \HR$ we write
\begin{equation}\label{part2.2:f_av_dfn}
\Xi_\mathrm{\;av}(f):=\frac{1}{\cD(n,d)}\; \int_{(f,v,\eta)\in \pi^{-1}(f)} \Xi(f,v,\eta) \;\d(f,v,\eta).
\end{equation}
The meaning of $\Xi_\mathrm{\;av}$ is explained in the following lemma.
\begin{lemma}\label{forward_distribution}
Let $\rho$ denote the density of the probability distribution that is obtained by choosing $f\sim \mathrm{Unif}\,\HS(\cH^n)$ and then choosing $(v,\eta)$ uniformly at random from the eigenpairs of $f$ (as in the beginning of \cref{sec:LVEALH}). Then
  $\mean_{(f,v,\eta)\sim \rho} \Xi(f,v,\eta) = \mean_{f\sim \mathrm{Unif}\,\HS(\cH^n)} \Xi_\mathrm{\;av}(f).$
\end{lemma}
\begin{proof}
The fiber $
   \pi_1^{-1}(f)=\set{(f,v,\eta)\in  \hV_\HS}$
over $f\not\in \Sigma$ consists of $\cD(n,d)$ points. Hence the density of the uniform distribution on that fiber equals $\cD(n,d)^{-1}$. Moreover, $\mathrm{Unif}\,\HS(\cH^n)$ is the forward distribution of $\rho$ under the projection $\hV_\HS \to \polspace$. this concludes the proof.
\end{proof}
\subsection{The condition number has small expectation}
This section is crucial for the proof of \cref{Theorem_4.6} presented in \cref{sec:avg_anal}.

\cref{Theorem_4.4} shows that the average number of steps of algorithm LVEALH is closely connected to the expectation of the condition number $\mu(f,v,\eta)$. By \cref{condition_for_eigenpairs} the expectation of $\mu(f,v,\eta)$ is bounded by the expectation of~$\hatmu(f,v,\eta)$. We prefer working with $\hatmu(f,v,\eta)$ because of the possibility of making a relative condition number out of it; see \cref{nu_def}.As in~\cref{part2.2:f_av_dfn2} we write
\begin{equation}\label{hatmu_av}(\hatmu^2)_{\mathrm{av}}(f):= \frac{1}{2\pi\cD(n,d)} \int_{\substack{(v,\eta)\in\myspaceS:\\ f(v)=\eta^{d-1}v}} \,\hatmu(f,v,\eta)^2\, \d(v,\eta)\end{equation}
\begin{prop}\label{mu_expectation0}
Let $\mathrm{Unif}\,\HS(\polspace)$ denote the uniform distribution on the unit sphere in $\polspace$. We have that
$$\mean\limits_{f\sim\mathrm{Unif}\,\HS(\polspace)}   (\hatmu^2)_{\mathrm{av}}(f) < \frac{80\,nN}{d} .$$
\end{prop}
The first task on the way to prove \cref{mu_expectation0} is to use  \cite[Corollary 2.23]{condition} to replace $\mathrm{Unif}\,\HS(\polspace)$ by the standard normal distribution $N(\polspace)$. The advantage of the latter is the independence of the coefficients of~$f\sim N(\polspace)$. Recall from \cref{nu_def} that we have put
\begin{equation}\label{nu_def2}
\hatmu_{\mathrm{rel}}(f,v,\eta):=
\Norm{\left[\begin{smallmatrix} \Norm{f} I_{n}&0\\0&\Norm{f}^\frac{d-2}{d-1}\end{smallmatrix}\right] \big(\restr{\deriv{\cF_f}{(v,\eta)}}{v^\perp\times \HC}\big)^{-1}},
\end{equation}
if $\restr{\deriv{\cF_f}{(v,\eta)}}{v^\perp\times \HC}$ is invertible and $\hatmu_{\mathrm{rel}}(f,v,\eta):=\infty$, otherwise.  By construction, for $\Norm{f}=1$ the relative and the absolute condition number coincide. This implies
  \begin{equation}\label{abc}
  \mean\limits_{q\sim\mathrm{Unif}\;\HS(\polspace)}   (\hatmu^2)_{\mathrm{av}}(q) = \mean\limits_{q\sim\mathrm{Unif}\;\HS(\polspace)}   (\hatmu_\mathrm{rel}^2)_{\mathrm{av}}(q)
  \end{equation}
For general $f$ we have the following inequality.
\begin{lemma}\label{mutilde}
For all $f\in\polspace$ we have $$(\hatmu_\mathrm{rel}^2)_{\mathrm{av}}(f)  \leq \max\set{\Norm{f}^2,\Norm{f}^\frac{2(d-2)}{d-1}} (\hatmu_{\mathrm{rel}}^2)_{\mathrm{av}}(f) .$$
\end{lemma}
\begin{proof}
Use \cref{nu_def2} the submultiplicativity of the spectral norm.
\end{proof}
From \cref{2dmu} (5) we see that $(\hatmu_{\mathrm{rel}}^2)_{\mathrm{av}}(f)$ is scale invariant. We further get from \cref{2dmu}~(4), that $(\hatmu_{\mathrm{rel}}^2)_{\mathrm{av}}(f)$ is unitarily invariant. By \cite[Corollary 2.23]{condition} we can write
\begin{equation}\label{abc2}
\mean\limits_{q\sim\mathrm{Unif}\;\HS(\polspace)}   (\hatmu_{\mathrm{rel}}^2)_{\mathrm{av}}(q)=\mean\limits_{q\sim N_{\HC}^{< \sqrt{2N}}(\polspace)}  (\hatmu_{\mathrm{rel}}^2)_{\mathrm{av}}(q),
\end{equation}
where $N_{\HC}^{< \sqrt{2N}}(\polspace)$ is the \emph{truncated normal distribution} \cite[(17.25)]{condition}. Using \cref{mutilde} we get
	\begin{equation}\label{abc3}\mean\limits_{q\sim N_{\sqrt{2N}}(\polspace)}  (\hatmu_{\mathrm{rel}}^2)_{\mathrm{av}}(q) \leq 2N \mean\limits_{q\sim N_{\HC}^{< \sqrt{2N}}(\polspace)} (\hatmu^2)_{\mathrm{av}}(q) \leq 4N \mean\limits_{q\sim N(\polspace)} (\hatmu^2)_{\mathrm{av}}(q);
  \end{equation}
the second inequality by \cite[Lemma 17.25]{condition}. \cref{mu_expectation0} is now obtained by combining \cref{abc}, \cref{abc2} and \cref{abc3} with the following result.
\begin{prop}\label{halloxyz} We have
$$\mean\limits_{q\sim N(\polspace)}  (\hatmu^2)_{\mathrm{av}}(q)\leq \frac{20n}{d}.$$
\end{prop}
\begin{proof}
In the following we ease notation by writing for the average condition number
  $$\hatmu_{\mathrm{av}}^2(q):=(\hatmu^2)_{\mathrm{av}}(q).$$
By \cref{part2.2:integration} we have
	\begin{equation}\label{xyzabc}\mean\limits_{q\sim N(\polspace)}  \hatmu_{\mathrm{av}}^2(q)=\frac{\mathrm{vol}\;\HS(\HC^n)}{2\pi^{n+1}\cD(n,d)} \int_{\eta\in\HC} (d-1)^2\; \widehat{E}(\eta)\; \norm{\eta}^{2(d-2)}\; e^{-\norm{\eta}^{2(d-1)}} \d \eta,\end{equation}
where
	\begin{equation}\label{gammaeq100}
	\widehat{E}(\eta):=E(\eta^{d-1})= \mean\limits_{A,a,h} \hatmu(q,e_1,\eta)^2 \norm{\det(\sqrt{d}A-\eta^{d-1}I_{n-1})}^2
	\end{equation}
and where $A\sim N(\HC^{(n-1)\times(n-1)})$, $a\sim N(\HC^{n-1})$, $h\sim N(R(e_1)^n)$ are independent, and where
\begin{equation}\label{q_identity}
q= \eta^{d-1}X_1^d e_1 + \sqrt{d}X_1^{d-1}\begin{bmatrix} a^T\\ A\end{bmatrix} (X_2,\ldots,X_n)^T + h.\end{equation}
\vspace{-1cm}
\begin{claim} We have
	\begin{align*}
	&\widehat{E}(\eta)(d-1)^2 \norm{\eta}^{2(d-2)}
	\leq n!d^{n-1} \big(1+(d-1)\norm{\eta}^{2(d-2)}\big) \sum_{k=0}^{n-1} \frac{1}{k!} \left(\frac{\norm{\eta}^{d-1}}{d}\right)^{2k}.
	\end{align*}
\end{claim}
If the claim where true, we would have
\begin{equation*}
\mean\hatmu_{\mathrm{av}}^2(q) \leq \frac{n!d^{n-1}\mathrm{vol}\;\HS(\HC^n)}{2\pi^{n+1}\cD} \int_{\eta\in \HC} (1+(d-1)\norm{\eta}^{2(d-2)}) e^{-\norm{\eta}^{2(d-1)}} \sum_{k=0}^{n-1} \frac{\norm{ \eta}^{2(d-1)k}}{k! d^k}\,\d \eta,
\end{equation*}
where $\cD:=\cD(n,d)=d^n-1$. Observe that the integrand is independent of the argument of $\eta$. Substituting $r:=\norm{\eta}$, so that $\d\eta = r\,\d r\, \d \theta$, shows
	\begin{equation*}
	\mean\limits_{q}\mu_{\mathrm{av}}^2(q)\leq \frac{(n-1)!nd^{n-1}\mathrm{vol}\,\HS(\HC^n)}{\pi^{n}\cD} \int_{r\geq 0} (r+(d-1)r^{2(d-2)+1}) e^{-r^{2(d-1)}} \sum_{k=0}^{n-1} \frac{r^{2(d-1)k}}{k! d^k}\d r.
	\end{equation*}
Making the change of variables $t:=r^{2(d-1)}$, interchanging summation and integration and using the fact that $\mathrm{vol}\;\HS(\HC^n) = 2\pi^n/(n-1)!$ yields
	\begin{align}
	\mean\limits_{q\sim N(0,\frac{1}{2}I)}\mu_{\mathrm{av}}^2(q)
	\leq\; &\frac{ nd^{n-1}}{\cD(d-1)} \sum_{k=0}^{n-1}\;\int_{t\geq 0} \left(t^{\frac{1}{d-1}-1}+(d-1)\right)  \frac{t^{k}e^ {-t}}{k! d^k}\;\d t\nonumber\\
	\;=\;&\frac{ nd^{n-1}}{\cD(d-1)} \sum_{k=0}^{n-1}\frac{\Gamma\big(k+\frac{1}{d-1}\big) + (d-1)\, \Gamma\big(k+1\big)}{d^k k!}\nonumber\\
\;=\;&\frac{ nd^{n-1}}{\cD(d-1)}\left( \sum_{k=0}^{n-1}\frac{\Gamma\big(k+\frac{1}{d-1}\big)}{d^k k!}+ (d-1)\sum_{k=0}^{n-1}\frac{1}{d^k}\right)\label{gammaeq0}
\end{align}
For $0<\varepsilon<1$ and $k\geq 2$ we have $\Gamma(k+\varepsilon)\leq \Gamma(k+1)=k!$ and hence
	\begin{align}
	\sum_{k=0}^{n-1}\frac{\Gamma\big(k+\frac{1}{d-1}\big)}{d^k k!} &\leq \;\Gamma\left(\frac{1}{d-1}\right) +\frac{\Gamma\big(1+\frac{1}{d-1}\big)}{d} + \sum_{k=2}^n\frac{1}{d^k}\nonumber\end{align}
Since $d\geq 2$ this is less or equal to
\begin{align*}
\Gamma\left(\frac{1}{d-1}\right) +\frac{\Gamma\big(1+\frac{1}{d-1}\big)}{d} + \sum_{k=2}^n\frac{1}{2^k}\;
	&\leq\;\Gamma\left(\frac{1}{d-1}\right) +\frac{\Gamma\big(1+\frac{1}{d-1}\big)}{d} +\frac{1}{2}\nonumber\\
	&\leq\;  (d-1)\frac{\sqrt{\pi}}{2} +\frac{\sqrt{\pi}}{2d} + \frac{1}{2}\; \leq \;3d;
\end{align*}
the last inequality by \cref{gamma_lemma}. Using that $d\geq 2$ we also have $$ (d-1)\sum_{k=0}^{n-1}\frac{1}{d^k} \leq d \sum_{k=0}^{n-1}\frac{1}{2^k} \leq 2d.$$ We have now found upper bounds for both summands in \cref{gammaeq0}. They imply that
	$$\mean\limits_{q\sim N(0,\frac{1}{2}I)}\mu_{\mathrm{av}}^2(q)
	\leq \frac{ 5nd^{n}}{\cD(d-1)} =5n\, \frac{d^{n}}{(d^n-1)(d-1)}\, \stackrel{d\geq 2}{\leq}\, \frac{20n}{d}.$$
This shows the assertion of \cref{halloxyz}. It remains to prove the claim.

\textit{Proof of the claim.}
The following proof follows ideas that we have found in \cite[Section~7]{Armentano2015a}. Recall from \cref{q_identity} the functional dependence of $q$ on $A,a,h$ and note that, by \cref{jacobian2},
\begin{align*}
  \deriv{\cF_q}{(e_1,\eta)}
  =\begin{bmatrix}
  (d-2)\eta^{d-1} & a^T & -(d-1)\eta^{d-2} \\
  0 & A-\eta^{d-1}I_{n-1} & 0\end{bmatrix}
\end{align*}
so that we have
\begin{equation}\label{mu_identification}
\hatmu(q,e_1,\eta)=
\Norm{\restr{\begin{bmatrix}
(d-2)\eta^{d-1} & \sqrt{d}a^T & -(d-1)\eta^{d-2} \\
0 & \sqrt{d}A-\eta^{d-1}I_{n-1} & 0\end{bmatrix}}{e_1^\perp\times\HC}^{-1}}^2.
\end{equation}
By \cref{gammaeq100} and \cref{mu_identification}, we now have
	\begin{equation*}
\widehat{E}(\eta)=\mean\limits_{A,a,h} \Norm{\restr{\begin{bmatrix}
   (d-2)\eta^{d-1} & \sqrt{d}a^T & -(d-1)\eta^{d-2} \\
   0 & \sqrt{d}A-\eta^{d-1}I_{n-1} & 0\end{bmatrix}}{e_1^\perp\times\HC}^{-1}}^2 \norm{\det(\sqrt{d}A-\eta^{d-1}I)}^2
	\end{equation*}
Let us put $b:=\sqrt{d}a$ and $B:=\sqrt{d}A-\eta^{d-1}I$ and
	\begin{equation}\label{derivative_expansion2}
	S :=\begin{bmatrix}   \sqrt{d} a^T & -(d-1)\eta^{d-2} \\ \sqrt{d}A -\eta^{d-1}I & 0\end{bmatrix} =\begin{bmatrix}   b^T & -(d-1)\eta^{d-2} \\ B & 0\end{bmatrix}.
	\end{equation}
Note that $\det(S)=-(d-1)\eta^{d-2}\det(B)$, so that
\begin{equation*}
  \widehat{E}(\eta) \;(d-1)^2 \norm{\eta}^{2(d-2)} = \mean\limits_{A,a,h} \; \Norm{S^{-1}}^2\norm{\det S}^2.
\end{equation*}
We use $\Norm{S^{-1}}^2\leq \Norm{S^{-1}}^2_F$, where $\Norm{\cdot}^2_F$ is the Frobenius norm. Moreover, $S$ is independent of $h$ so we may omit it in the expectation. This yields
\begin{equation}
	\label{z3}\widehat{E}(\eta) \;(d-1)^2 \norm{\eta}^{2(d-2)} \leq \mean\limits_{A,a} \; \Norm{S^{-1}}_F^2\norm{\det S}^2.
	\end{equation}
We enumerate the entries of $S$ with indices $0\leq i,j\leq n-1$, the entries of $B$ are enumerated with indices $1\leq i,j\leq n-1$.
Let $S^{i,j}\in\HC^{(n-1)\times(n-1)}$, $0\leq i,j\leq n-1$, denote the matrix that is obtained by removing from $S$ the $i$-th row and the $j$-th column and define $B^{i,j},1\leq i,j\leq n-1$, correspondingly. By Cramer's rule we have $$\Norm{S^{-1}}_F^2\norm{\det S}^2 = \sum_{0\leq i,j\leq n-1} \norm{\det S^{i,j}}^2.$$
We deduce from \cref{derivative_expansion2} that
\begin{equation*}
\norm{\det S^{i,j}}^2 =
\begin{cases}
0, &\text{if } i=0,\, j <n-1\\
\norm{\det B}^2, &\text{if } i=0,\,j=n-1\\
(d-1)^2\eta^{2(d-2)} \norm{\det B^{i,j}}^2, &\text{if } i>0,\,j<n-1\\
 \norm{\det B(i:b)}^2, &\text{if } i>0,\, j=n-1
\end{cases},
\end{equation*}
where $B(i:b)$ is the matrix that is obtained from $B$ by replacing the $i$-th row of $B$ with $b$. We have $\mathrm{Var}(B_{i,j})= d\;\mathrm{Var}(A_{i,j})=d$ for $1\leq i,j\leq n-1$, so that
\begin{equation*}
d\sum_{j=1}^{n-1}  \mean\norm{\det B^{i,j}}^2=\sum_{j=1}^{n-1}  \mathrm{Var}(B_{i,j}) \mean\norm{\det B^{i,j}}^2.
\end{equation*}
From \cref{equation_random_matrix} we get for fixed $1\leq i \leq n-1$
\begin{equation}\label{r5}
\sum_{j=1}^{n-1}  \mathrm{Var}(B_{i,j}) \mean\norm{\det B^{i,j}}^2 \leq \mean \norm{\det B}^2.
\end{equation}
We moreover have $\mean b =0$ and for $1\leq j\leq n-1$ we have $\mathrm{Var}(b_j)= d$. Using \cref{equation_random_matrix} again, we obtain for $1\leq i \leq n-1$
\begin{align*}
\mean \norm{\det B(i:b)}^2 &= \mean \norm{\det B(i:0)}^2+\sum_{j=1}^{n-1} \mathrm{Var}(b_j)\mean\norm{\det B^{i,j}}^2
=  d \sum_{j=1}^{n-1} \mean\norm{\det B^{i,j}}^2.
\end{align*}
Combining this with \cref{r5} we get
$ \mean \norm{\det B(i:b)}^2\leq\mean\norm{\det B}^2.$
We conclude that
	\begin{align*}
	\mean \Norm{S^{-1}}_F^2\norm{\det S}^2
	\leq &\mean \norm{\det B}^2 + \frac{(n-1)(d-1)^2\norm{\eta}^{2(d-2)}}{d}\;\mean\norm{\det B}^2 + (n-1)\mean\norm{\det B}^2.
\end{align*}
The right-hand side of that inequality we bound by $n\left(1+(d-1)\eta^{2(d-2)}\right) \mean \norm{\det B}^2.$
We have
\begin{align*}
\mean \norm{\det B}^2 &= \mean \norm{\det(\sqrt{d}A-\eta^{d-1} I_{n-1})}^2
= d^{n-1}\mean \norm{\det\left(A-\frac{\eta^{d-1}}{\sqrt{d}} I_{n-1}\right)}^2
\end{align*}
and by \cref{auxiliary_lemma} (1) and (3) we have
 \begin{equation*}\mean \norm{\det\Big(A-\frac{\eta^{d-1}}{\sqrt{d}} I_{n-1}\Big)}^2 =(n-1)! \sum_{k=0}^{n-1} \frac{1}{k!}\, \left(\frac{\norm{\eta}^{d-1}}{d}\right)^{2k}.
 \end{equation*}
Combining everything we get
$$\mean \Norm{S^{-1}}_F^2\norm{\det S}^2\leq  n! d^{n-1}\left(1+(d-1)\eta^{2(d-2)}\right) \sum_{k=0}^{n-1} \frac{1}{k!}\, \left(\frac{\norm{\eta}^{d-1}}{d}\right)^{2k}.$$
Plugging this into \cref{z3} implies the claim.
\end{proof}
\subsection{Trigonometry}
The following proposition basically says that for pairs of points $(v_1,\eta_2)$, $(v_2,\eta_2)$ in the grey area of \cref{fig1}, that are close to each other, the points $v_1,v_2$ are also close to each other.
\begin{prop}\label{real_geometry_prop}
Let $m=m_1+m_2$, and $x:=(x_1,x_2)\in(\HR^{m_1}\backslash\set{0})\times\HR^{m_2}$ be a point, such that $\Norm{x_2} \leq \Norm{x_1}$. Let $y:=(y_1,y_2)\in (\HR^{m_1}\times\HR^{m_2})\backslash\set{0}$ and put $\delta:=d_\HS(x,y)$
\begin{enumerate}
\item If $\delta < \pi/4$, then $y_1\neq 0$.
\item Suppose that $\delta < \frac{\pi}{16}$ and $x_1,y_1\in\HS(\HR^{m_1})$. Then $\Norm{x-y}\leq  \delta\sqrt{8}.$
\end{enumerate}
\end{prop}
\begin{proof}
Let us first prove the first part. Without restriction we may assume that $\Norm{x}=\Norm{y}=1$. Consider the map
	\begin{align*}
	\Psi: \HS(\HR^{m_1})\times \HS(\HR^{m_2}) \times [0,\pi/2]&\to \HS(\HR^{m_1+m_2}),\;		(u,w,\phi)\mapsto \left(\cos(\phi)u,\sin(\phi)v\right)
	\end{align*}
One easily shows that $\Psi$ is surjective. Let $\Psi(u_x,w_x,\phi_x)=x$ and $\Psi(u_y,w_y,\phi_y)=y$. Then,
	\begin{align*}\cos(\delta) &= \cos(\phi_x)\cos(\phi_y)\langle u_x,u_y\rangle + \sin(\phi_x)\sin(\phi_y)\langle w_x,w_y\rangle \\
	&\leq \cos(\phi_x)\cos(\phi_y)
	 + \sin(\phi_x)\sin(\phi_y) = \cos(\phi_x-\phi_y)
	\end{align*}
and hence $\delta\geq \norm{\phi_x-\phi_y}$. From our assumption $\Norm{x_2}\leq \Norm{x_1}$ follows $0\leq \phi_x\leq \pi/4$, which implies $0\leq \phi_y \leq \delta+\phi_x< \pi/2$. This implies $\Norm{y_1}=\cos(\phi_y)\neq 0$. The first assertion follows from this.

We now prove the second claim. Write  $\Psi(u_x,w_x,\phi_x)=\Norm{x}^{-1} x$ and we write  $\Psi(u_y,w_y,\phi_y)=\Norm{y}^{-1}y$. By assumption, we have $\cos(\phi_x)\neq 0$ and, by the first part, we have $\cos(\phi_y)\neq 0$. Moreover, $\Norm{x_1}=\Norm{y_1}=1$ implies $\Norm{x}^2  = 1 + \tan^2(\phi_x),$ and $\Norm{y}^2  = 1 + \tan^2(\phi_y)$. The law of cosines yields $$\Norm{x-y}^2 = \Norm{x}^2 + \Norm{y}^2 - 2\Norm{x}\Norm{y} \cos(\delta).$$ From this we obtain
	$$\Norm{x-y}^2 = 2  + \tan^2\phi_x + \tan^2\phi_y -2 \sqrt{(1+\tan^2\phi_x)(1+\tan^2\phi_y)} \cos(\delta).$$
Writing $\phi_y = \phi_x+\delta$ we obtain $\Norm{x-y}^2$ as a function of $\phi_x$ and $\delta$, which we denote by
$$L:[0,\frac{\pi}{4}]\times [-\frac{\pi}{16},\frac{\pi}{16}] \to \HR_{\geq 0},\: (\phi_x,\delta)\mapsto \Norm{x-y}^2.$$
A computer based calculation reveals that $\frac{\d L}{\d \phi_x} >0$. This implies
	$$\Norm{x-y}^2 \leq L(\tfrac{\pi}{4},\delta) = 3 + \tan(\pi/4+\delta)^2 -2 \sqrt{2(1+\tan(\pi/4+\delta)^2)} \cos(\delta).$$
Another computer based calculation shows $L(\tfrac{\pi}{4},0)=0$ and that $0$ is a global minimum of the function $\delta\mapsto 8\delta^2 - L(\tfrac{\pi}{4},\delta)$.
	\end{proof}
\subsection{Higher derivative estimate}\label{condition_number_thm_proof}
In this section we prove \cref{Theorem_4.3}. We will need the following lemma.
\begin{lemma}\label{condition_number_thm_lemma}
For $d\geq 2$ and $2\leq k\leq d$ we have that $$\frac{(d-1)!}{k!(d-k)!}\leq \left(\frac{d}{2}\right)^{k-1}.$$
\end{lemma}
\begin{proof}
Combine $k!\geq 2^{k-1}$ with $\frac{(d-1)!}{(d-k)!} \leq d^{k-1}$.
\end{proof}
The following is similar to the proof of \cite[Theorem 16.1]{condition}.
\begin{proof}[Proof of \cref{Theorem_4.3}]
We choose a representative $(v,\eta)\in\HS(\HC^n)\times\HC$, so that, by assumption, $(f,v,\eta)\in\hatV_\HS$. From \cref{useful_lemma2} we get $\Norm{v,\eta}\leq \sqrt{2}$. Moreover, for any $2\leq k\leq d$, using the submultiplicativity of the spectral norm, we have
 \begin{align}
 \Norm{\frac{1}{k!} \big(\restr{\deriv{\cF_f}{(v,\eta)}}{(v,\eta)^\perp}\big)^{-1} \cdot \derivk{\cF_f}{(v,\eta)}}
 &\leq\frac{1}{k!}\; \Norm{\restr{\big(\deriv{\cF_f}{(v,\eta)}}{(v,\eta)^\perp}\big)^{-1} } \Norm{ \derivk{\cF_f}{(v,\eta)}}\nonumber\\
 &= \frac{1}{k!}\;\mu(f,v,\eta) \Norm{ \derivk{\cF_f}{(v,\eta)}}.\label{gamma_to_condition0}
 \end{align}
 By the triangle inequality we have
	\begin{equation}
	\Norm{ \derivk{\cF_f}{(v,\eta)}} \leq \Norm{ \derivk{f}{v}} + \Norm{\derivk{\ell^{d-1}X}{(v,\eta)}}.\label{gamma_to_condition1}
	\end{equation}
Using $\Norm{f}=\Norm{v}=~1$ we get from \cite[Prop. 16.48]{condition}
\begin{equation}\label{gamma_to_condition2}
\Norm{\derivk{f}{v}} \leq \frac{d!}{(d-k)!}
\end{equation}
One easily sees that $$\Norm{\derivk{\ell^{d-1}X}{(v,\eta)}} \leq \Big(\sum_{i=1}^n \Norm{\derivk{\ell^{d-1}X_1}{(v,\eta)}}^2\Big)^\frac{1}{2}.$$ For $w_1,\ldots,w_k\in\HC^n $ let $w_i^j$ be the $j$-th entry of $w_i$ and denote partial derivatives by $\frac{\d}{\d X_i}$. Then
\begin{align*}
	\Norm{\derivk{\ell^{d-1}X_i}{(v,\eta)}}&= \max\limits_{w_1,\ldots,w_k\in\HS(\HC^n)} \;\norm{\sum_{1\leq i_1,\ldots,i_k\leq n} \frac{\d^k (\ell^{d-1}X_i)}{\d X_{i_1}\ldots\d X_{i_k}}(v,\eta) \prod_{j=1}^k w_j^{i_j}}\\
	&\leq \sum_{1\leq i_1,\ldots,i_k\leq n} \norm{\frac{\d^k (\ell^{d-1}X_i)}{\d X_{i_1}\ldots\d X_{i_k}}(v,\eta)}
	\end{align*}
The only $k$-th order derivatives of $\ell^{d-1}X_i$, that are non-zero, are $\ell^{d-1-k} X_i~\prod_{i=0}^{k-1}(d-1-i)$ and $ \ell^{d-k}\prod_{i=0}^{k-2} (d-1-i)$. This implies
	\begin{align*}
	\Norm{\derivk{\ell^{d-1}X}{(v,\eta)}} &\leq \frac{ (d-1)!}{(d-1-k)!} \; \norm{\eta}^{d-1-k} \norm{v_i} +  k \; \frac{(d-1)!}{(d-k)!} \; \norm{\eta}^{d-k}\leq \frac{ d!}{(d-k)!};
	\end{align*}
where we have used from \cref{useful_lemma2} that $\norm{\eta}\leq\Norm{v}=1$. This shows $\Norm{\derivk{\ell^{d-1}X}{(v,\eta)}} \leq \sqrt{n} \;\frac{ d!}{(d-k)!}$. Plugging this and \cref{gamma_to_condition2} into \cref{gamma_to_condition1} we get
$$\Norm{ \derivk{\cF_f}{(v,\eta)}} \leq (\sqrt{n}+1)\;\frac{d!}{(d-k)!}.$$
Using \cref{gamma_to_condition0} we obtain
	 \begin{align*}
 \Norm{v,\eta}^{k-1}\; \Norm{ \big(\deriv{\cF_f}{(v,\eta)}|_{(v,\eta)^\perp}\big)^{-1} \; \frac{\derivk{\cF_f}{(v,\eta)}}{k!}}
 &\leq  \sqrt{2}^{k-1}\;\mu(f,v,\eta)\; \frac{(\sqrt{n}+1)\;d!}{(d-k)!\;k!}\\
 	 &\leq  \sqrt{2n}^{k-1}\; \mu(f,v,\eta) \; \frac{2\;d!}{(d-k)!k!}\\
 	  &= \sqrt{2n}^{k-1}\; 2d\; \mu(f,v,\eta) \; \frac{(d-1)!}{(d-k)!k!}.
 \end{align*}
By \cref{condition_number_thm_lemma} we have $\frac{(d-1)!}{k!(d-k)!}\leq \left(\frac{d}{2}\right)^{k-1}$. Furthermore, by \cref{2dmu} (2) we have $1\leq 2d\,\mu(f,v,\eta)$, so that $2d\,\mu(f,v,\eta)\leq (2d\,\mu(f,v,\eta))^{k-1}$. Hence
$$\Norm{v,\eta}\; \Norm{\frac{1}{k!} \big(\deriv{\cF_f}{(v,\eta)}|_{(v,\eta)^\perp}\big)^{-1} \derivk{\cF_f}{(v,\eta)}}^\frac{1}{k-1}
 \leq \mu(f,v,\eta)\; d^2 \,\sqrt{2n},$$
from which the claim follows.
\end{proof}
\subsection{A Lipschitz estimate for the condition number}\label{lipschitz_sec}
We will have to quantify the behavior of the condition number $\mu(f,v,\eta)$ under small pertubations of both $f$ and $(v,\eta)$. In what follows we define the functions $\vartheta(\varepsilon):=1-(1-\varepsilon)^{-2}+\cos(\tfrac{\varepsilon}{4})$ and, as in \cref{Theta}, $\Theta(\varepsilon)= \vartheta(\varepsilon)-\varepsilon .$ Observe that
\begin{equation}\label{theta_ineq}
\forall \,0\leq\varepsilon\leq \frac{1}{5}:\; 0<\Theta(\varepsilon)\leq 1\quad\text{and}\quad \varepsilon<\vartheta(\varepsilon).
\end{equation}
This section is heavily inspired by \cite[Section 16.8]{condition}. The main results here are \cref{lipschitz1} and \cref{lipschitz3}
In order to prove them we will need the following lemma, which is similiar to \cite[Lemma~16.41]{condition}. Recall from \cref{gamma_def} the definition
$$\gamma(f,v,\eta):=\Norm{v,\eta}\;\max\limits_{k\geq 2} \Norm{\frac{1}{k!}\; \big(\deriv{\cF_f}{(v,\eta)}|_{(v,\eta)^\perp}\big)^{-1} \; \derivk{\cF_f}{(v,\eta)}}^\frac{1}{k-1}.$$
 and from \cref{u_and_delta} the definition $\Psi_\delta(u)=(1+\cos\delta)(1-u)^2-1,\; u,\delta\in\HR.$
\begin{lemma}\label{lipschitz0}
Let $(f,v,\eta)\in\hatW_\HS$ and let $(w,\xi)\in\myspaceS$. Put
$$\delta:=d_\HS((w,\xi),(v,\eta)), \quad \text{and}\quad u:=\delta\gamma(f,v,\eta)\sqrt{8}.$$
If $\delta<\frac{\pi}{16}$ and $ u <1$, then
	$$\Norm{ \big(\deriv{\cF_f}{(v,\eta)}|_{(v,\eta)^\perp}\big)^{-1}\; \deriv{\cF_f}{(w,\xi)}|_{(w,\xi)^\perp}}\leq  \frac{1}{(1-u)^2}.$$
If further $\Psi_\delta(u)>0$, then $\deriv{\cF_f}{(w,\xi)}|_{(w,\xi)^\perp}$ is invertible and we have
	$$\Norm{\deriv{\cF_f}{(w,\xi)}|_{(w,\xi)^\perp}^{-1} \;  \deriv{\cF_f}{(v,\eta)}|_{(v,\eta)^\perp}}\leq  \frac{(1-u)^2}{\Psi_\delta(u)}.$$
\end{lemma}
\begin{proof}
Recall that $\cF_f$ is a map $\HC^n\times \HC\to \HC^n$. We identify $\HC^{n+1}\cong \HC^n\times \HC$ and we denote by $\cL(\HC^{n+1},\HC^n)$ the vector space of $\HC$-linear maps $\HC^{n+1}\to\HC^n$, compare \cite[(15.2)~and~Lemma 15.9]{condition}. For any $f\in\cH$ we define
	$$\deriv{\cF_f}{}: \HC^{n+1} \to \cL(\HC^{n+1},\HC^n),\; (w,\xi)\mapsto \deriv{\cF_f}{(w,\xi)},$$
and write  $$\mathrm{D}^2_{(w,\xi)}\cF_f :=\deriv{(\deriv{\cF_f}{})}{(w,\xi)} \in \cL(\HC^{n+1},\cL(\HC^{n+1},\HC^n)).$$ Inductively we define $\derivk{\cF_f}{(w,\xi)}$. By construction, for fixed $k$ the map $\derivk{\cF_f}{(w,\xi)}$ is a \emph{multilinear} map $(\HC^{n+1})^{\times (k-1)} \to \HC^n$. We abbreviate
$$\derivk{\cF_f}{(w,\xi)} x^{k-1}:=\derivk{\cF_f}{(w,\xi)}(x,\ldots,x).$$
The Taylor expansion of $\deriv{\cF_f}{(w,\xi)}|_{(w,\xi)^\perp}$ around $(v,\eta)$ can then be expressed as
	$$\deriv{\cF_f}{(w,\xi)}|_{(w,\xi)^\perp} = \sum_{k=1}^d \frac{1}{(k-1)!}\;  \derivk{\cF_f}{(v,\eta)}|_{(w,\xi)^\perp} (w-v,\xi-\eta)^{k-1},$$
see also \cite[Lemma 16.41]{condition}. Then $(\deriv{\cF_f}{(v,\eta)}|_{(v,\eta)^\perp})^{-1} \;  \deriv{\cF_f}{(w,\xi)}|_{(w,\xi)^\perp} = P + B,$
where,
	\begin{align}P&=  \big(\deriv{\cF_f}{(v,\eta)}|_{(v,\eta)^\perp}\big)^{-1} \;  \deriv{\cF_f}{(v,\eta)}|_{(w,\xi)^\perp}\\
B&= \sum_{k=2}^d \frac{1}{(k-1)!} \big(\deriv{\cF_f}{(v,\eta)}|_{(v,\eta)^\perp}\big)^{-1} \;  \derivk{\cF_f}{(v,\eta)}|_{(w,\xi)^\perp} (w-v,\xi-\eta)^{k-1}.\label{B}
\end{align}
By \cref{projection_lemma} we have that $\Norm{P}\leq 1$. If in \cref{real_geometry_prop} we put $x=(v,\eta),y=(w,\xi)$, we obtain $\Norm{w-v,\xi-\eta} \leq \delta\sqrt{8}$.  Combining with \cref{B} we obtain
\begin{align*}
\Norm{B}&\leq\sum_{k=1}^d  k \Norm{\frac{1}{k!}\, \big(\deriv{\cF_f}{(v,\eta)}|_{(v,\eta)^\perp}\big)^{-1} \;  \derivk{\cF_f}{(v,\eta)}(w-v,\xi-\eta)^{k-1}}\\
&\leq \sum_{k=1}^d  k \Norm{\frac{1}{k!}\, \big(\deriv{\cF_f}{(v,\eta)}|_{(v,\eta)^\perp}\big)^{-1} \;  \derivk{\cF_f}{(v,\eta)}}\,\Norm{w-v,\xi-\eta}^{k-1}\\
&\leq \sum_{k=1}^d  k \Norm{v,\eta}^\frac{1}{k-1}\Norm{\frac{1}{k!}\big(\deriv{\cF_f}{(v,\eta)}|_{(v,\eta)^\perp}\big)^{-1} \derivk{\cF_f}{(v,\eta)}}\Norm{w-v,\xi-\eta}^{k-1}\\
&\leq \sum_{k=2}^d  k\left(\gamma(f,v,\eta)\,\delta\,\sqrt{8}\right)^{k-1} =   \sum_{k=2}^d  ku^{k-1}\leq (1- u)^{-2}-1;
\end{align*}
the first inequality by the triangle inequality, the second by \cref{spectral_norm} and the third because $\Norm{v}=1$. This implies
$$\Norm{\big(\deriv{\cF_f}{(v,\eta)}|_{(v,\eta)^\perp}\big)^{-1} \;  \deriv{\cF_f}{(w,\xi)}|_{(w,\xi)^\perp}} = \Norm{P+B} \leq \Norm{P}+\Norm{B} \leq (1-u)^{-2},$$
which is the first assertion.

It remains prove the second assertion. \cref{projection_lemma} tells us that $P$ is invertible and that $\Norm{P^ {-1}}\leq (\cos \delta)^{-1}.$ By our assumption $\Psi_\delta(u) > 0$ we have
$$\Norm{P^{-1}}\Norm{B} \leq (\cos\delta)^{-1}\; \left((1-u)^{-2}-1\right) < 1.$$
From \cite[Lemma 15.7]{condition} we get that $P+B$ is invertible and that
	$$\Norm{(P+B)^{-1}}\leq \frac{\Norm{P^{-1}}}{1-\Norm{B}\Norm{P^{-1}}} \leq \frac{(\cos \delta)^{-1}}{1- (\cos\delta)^{-1}\left((1-u)^{-2}-1\right)}.$$
One shows that the right hand side equals $ (1-u)^2/\Psi_\delta(u)$. This finishes the proof.
\end{proof}
The next proposition is of most importance in this section.
\begin{prop}\label{lipschitz1}
\begin{enumerate}
\item
Let $f,g\in\HS(\polspace)$ and $(v,\eta)\in \cP$. If
$$d\, \mu(f,v,\eta)\;d_\HS(f,g)\leq \varepsilon < 1,$$
then
$$\mu(g,v,\eta)\leq \frac{1}{1-\varepsilon}\; \mu(f,v,\eta).$$
\item
Let $(f,v,\eta)\in\hW_\HS$ and $(w,\xi)\in\cP$. If
$$4\;d^2 \sqrt{n}\; \mu(f,v,\eta)\; d_\HP((w,\xi),(v,\eta))\leq \varepsilon \leq \frac{1}{4},$$
then
$$(1-\varepsilon)^2\mu(f,v,\eta)\leq \mu(f,w,\xi)\leq \frac{1}{\vartheta(\varepsilon)}\; \mu(f,v,\eta).$$
\end{enumerate}
\end{prop}
\begin{proof}
We prove the first assertion. Let $(v,\eta)\in \HS(\HC^n)\times \HC$ be a representative for $(v,\eta)\in\cP$. By assumption we have $$d\, \mu(f,v,\eta) \,d_\HS(f,g)\leq \varepsilon < 1.$$ This implies that $\mu(f,v,\eta)<\infty$, which means that $\deriv{\cF_f}{(v,\eta)}|_{(v,\eta)^\perp}$ is invertible. We denote $A:= \deriv{\cF_f}{(wv,\eta)}|_{(v,\eta)^\perp}$, so that $\mu(f,v,\eta) = \Norm{A^{-1}}$.
Futher, we write $$\Delta:= [\deriv{(g-f)}{w},0]|_{(v,\eta)^\perp},$$ so $A + \Delta = \deriv{g}{(v,\eta)}|_{(v,\eta)^\perp}$. We have
	$$\Norm{\Delta}\leq \Norm{[\deriv{(g-f)}{v},0]}= \Norm{\deriv{(g-f)}{v}}\leq d  \Norm{g-f};$$
the last inequality by \cite[Lemma 16.46]{condition}. Since $\Norm{f}=\Norm{g}=1$, we have that $\Norm{f-g}\leq d_\HS(f,g)$. Therefore,
	$$\Norm{A^{-1}}\Norm{\Delta} \leq d\Norm{A^{-1}} \Norm{g-f}\leq d\;\mu(f,v,\eta)\; d_\HS(f,g) \leq \varepsilon < 1.$$
We can apply \cite[Lemma 15.7]{condition} to deduce that $A+\Delta$ is invertible and that
	\begin{equation*}
	\mu(g,v,\eta) = \Norm{(A+\Delta)^{-1}}
	\leq \frac{\Norm{A^{-1}}}{1-\Norm{\Delta}\Norm{A^ {-1}}}
	\leq \frac{1}{1-\varepsilon}\;\mu(f,v,\eta).
	\end{equation*}
This proves the first assertion.

To prove the second assertion for the h-eigenpairs $(v,\eta),(w,\xi)\in\cP$ we choose representatives $(v,\eta),(w,\xi)\in\HS(\HC^n)\times \HC$, such that $$\delta:=d_\HP((v,\eta),(\xi,\eta))=d_\HS((v,\eta),(\xi,\eta)).$$
Put $u:=\delta\gamma(f,v,\eta)\sqrt{8}$.  By \cref{Theorem_4.3} and by our assumption we have
	$$u\leq 4  \delta d^2\sqrt{n}  \,\mu(f,v,\eta)\leq \varepsilon\leq \frac{1}{4} <1.$$
Using that  $d\geq 2$ and from \cref{2dmu} (2) that $2d\mu(f,v,\eta)\geq 1$ we obtain $\delta\leq \varepsilon/4<\frac{\pi}{16}.$
Recall that $\Psi_\delta(u)/(1-u)^2 = 1+\cos\delta - (1-u)^{-2}.$
The right hand side is decreasing in $\delta$ and in $u$. This implies
$$\frac{\Psi_\delta(u)}{(1-u)^2}=1+\cos\delta - (1-u)^{-2} \geq 1+\cos (\varepsilon/4) - (1-\varepsilon)^{-2}=\vartheta(\varepsilon).$$ We get $$\frac{\Psi_\delta(u)}{(1-u)^2}\geq \vartheta(\varepsilon) > 0,$$ by \cref{theta_ineq}. Summarizing, we have  $\delta<\frac{\pi}{16}$, $u <1$ and $\Psi_\delta(u)>0$. We can therefore apply \cref{lipschitz0} to deduce that $\mu(f,w,\xi)<\infty$ and, using the submultiplicativity of the spectral norm, that
	\begin{align*}
	\mu(f,w,\xi)&= \Norm{\deriv{\cF_f}{(w,\xi)}_{(w,\xi)^\perp}^{-1}\; \deriv{\cF_f}{(v,\eta)}_{(v,\eta)^\perp}\; \deriv{\cF_f}{(v,\eta)}_{(v,\eta)^\perp}^{-1}}\\
	&\leq  \Norm{\deriv{\cF_f}{(w,\xi)}_{(w,\xi)^\perp}^{-1}\; \deriv{\cF_f}{(v,\eta)}_{(v,\eta)^\perp}}\; \Norm{\deriv{\cF_f}{(v,\eta)}_{(v,\eta)^\perp}^{-1}}\\
&\leq \frac{(1-u)^2}{\Psi_\delta(u)}\; \mu(f,v,\eta)\\ &\leq   \frac{1}{\vartheta(\varepsilon)}\; \mu(f,v,\eta);
	\end{align*}
the second-to-last inequality by \cref{lipschitz0}. This proves the first inequality.

Similiary we obtain
	\begin{align*}
	\mu(f,v,\eta)&= \Norm{\deriv{\cF_f}{(v,\eta)}_{(v,\eta)^\perp}^{-1}\; \deriv{\cF_f}{(w,\xi)}_{(w,\xi)^\perp}\; \deriv{\cF_f}{(w,\xi)}_{(w,\xi)^\perp}^{-1}}\\
&\leq \frac{1}{(1-u)^2}\; \mu(f,w,\xi)\\
 &\leq \frac{1}{(1-\varepsilon)^2}\; \mu(f,w,\xi).
	\end{align*}
This finishes the proof.
\end{proof}
Next we combine the statements from \cref{lipschitz1}.
\begin{cor}\label{lipschitz3}
Let $(f_1,v_1,\eta_1),(f_2,v_2,\eta_2)\in\hW_\HS$. If
	$$d\, \mu(f_1,v_1,\eta_1)\; \max\set{ d_\HS(f_1,f_2),  4\; d\sqrt{n}\; d_\HP((v_1,\eta_1),(v_2,\eta_2)) } \leq \varepsilon \leq \frac{1}{5},$$
then
	$$ \Theta(\varepsilon) \;\mu(f_1,v_1,\eta_1)\leq \mu(f_2,v_2,\eta_2) \leq \frac{1}{\Theta(\varepsilon)} \;\mu(f_1,v_1,\eta_1).$$
\end{cor}
\begin{proof}
We may apply \cref{lipschitz1}(2) to obtain
	\begin{equation}\label{lipschitz3.1}
	\mu(f_1,v_2,\eta_2)\leq \frac{1}{\vartheta(\varepsilon)} \; \mu(f_1,v_1,\eta_1).
	\end{equation}
This yields $d \, \mu(f_1,v_2,\eta_2)\;  d_\HS(f,g) \leq \varepsilon/\vartheta(\varepsilon)=:\varepsilon' < 1,$ by \cref{theta_ineq}. Applying \cref{lipschitz1}(1) yields
	\begin{equation}\label{lipschitz3.2}
	\mu(f_2,v_2,\eta_2) \leq \frac{1}{1-\varepsilon'} \;  \mu(f_1,v_2,\eta_2)
	\end{equation}
Combining \cref{lipschitz3.1} and \cref{lipschitz3.2} gives
	$$\mu(f_2,v_2,\eta_2) \leq \frac{1}{(1-\varepsilon')\vartheta(\varepsilon)} \; \mu(f_1,v_1,\eta_1) = \frac{1}{\vartheta(\varepsilon)-\varepsilon} \; \mu(f_1,v_1,\eta_1)$$
This yields the right inequality. If $\mu(f_1,v_1,\eta_1)\leq \mu(f_2,v_2,\eta_2)$, the left inequality is trivial. If $\mu(f_2,v_2,\eta_2)\leq \mu(f_1,v_1,\eta_1)$, then
	\begin{align*}
	&d \; \mu(f_2,v_2,\eta_2)\;   \max\set{ d_\HS(f_1,f_2),  4\; d\sqrt{n}\; d_\HP((v_1,\eta_1),(v_2,\eta_2)) }\\
	 \leq \;&d \; \mu(f_1,v_1,\eta_1)\; \max\set{ d_\HS(f_1,f_2),  4\; d\sqrt{n}\; d_\HP((v_1,\eta_1),(v_2,\eta_2)) }\\
   \leq\;&\;  \varepsilon .
	\end{align*}
and we may interchange the roles of $(f_1,v_1,\eta_1)$ and $(f_2,v_2,\eta_2)$ to conclude.
\end{proof}
\section{Complexity analysis}
In this section we prove the previously stated complexity analyses for the four algorithms EALH, LVEALH and LVEALHWS and Draw-from-$\rho^\star$; that is, we will prove \cref{Theorem_4.4}, \cref{Theorem_4.6}, \cref{Theorem_4.7} and \cref{Theorem_4.8}.
\subsection{Complexity of algorithm \textsc{E-ALH}}
\label{ALH-proof}
Recall from \cref{Theta} that we have defined the function $\Theta(\varepsilon)= 1-(1-\varepsilon)^{-2}+\cos(\tfrac{\varepsilon}{4})-\varepsilon.$
We need the following lemma to prove \cref{Theorem_4.4}.
\begin{lemma}\label{ALH-lemma}
Let $(q,v,\eta),(q',w,\xi)\in \hW_\HS$, $q\neq -q'$, such that the geodesic path $E\subset \HS(\cH^n)$ connecting $q$ and $q'$ does not intersect the extended eigendiscriminant variety $\cE$ (see \cref{prop_eigendiscriminant}). Let $\varepsilon < \frac{1}{5}$ and $\chi < 1$.  Then,
$$d_\HS(q,q') \leq  \frac{ \varepsilon\chi\Theta(\varepsilon)}{4 d^{2} \sqrt{n} \mu(q,v,\eta)^2}$$
implies that
 $$d_\HP((v,\eta),(w,\xi)) \leq  \frac{\varepsilon\chi}{4d^{2} \sqrt{n}  \mu(q,v,\eta)}.$$
\end{lemma}
\begin{proof}
Let $E=\cset{q_\tau}{0\leq \tau \leq 1}.$ If $E$ does not cross $\cE$, then, there exists a unique lifted path $L=\cset{(q_\tau,v_\tau,\eta_\tau)}{0\leq \tau \leq 1}\subset \hatV$. \cref{useful_lemma2} assures that all h-eigenpairs in that path are indeed elements in $\cP$, so that the path is a subset of $\hatV$ (it could be a priori that one of the h-eigenpairs is continuated towards the trivial solution $[0:1]$).

Suppose now that the assertion were false. Then there exists some $0\leq \tau^* < 1$, such that
	$$d_\HP((v,\eta),(v_{\tau^*},\eta_{\tau^*}))=\int_0^{\tau^*} \Norm{\dot{v}_\tau,\dot{\eta}_\tau}\d\tau  =  \frac{\varepsilon\chi}{4d^{2} \sqrt{n} \mu(q,v,\eta)}.$$
Let $0\leq \tau\leq \tau^*$. From \cref{2dmu} (3) we obtain $\Norm{\dot{v}_\tau,\dot{\eta}_\tau} \leq \mu(q_\tau,v_\tau,\eta_\tau) \Norm{\dot{q}_\tau}$ and hence
	$$ \frac{\varepsilon\chi}{4d^{2} \sqrt{n} \mu(q,v,\eta)}\leq \int_0^{\tau^*} \mu(q_\tau,v_\tau,\eta_\tau) \Norm{\dot{q}_\tau}\d\tau  .$$
Using $d_\HP((v,\eta),(v_\tau,\eta_\tau))\leq d_\HP((v,\eta),(v_{\tau^*},\eta_{\tau^*}))$ and $d_\HS(q,q_\tau)\leq d_\HS(q,q')$,
 that $2d\mu(q,v,\eta)\geq 1$ (see \cref{2dmu} (2)) and $\Theta(\varepsilon)\leq 1$ (see \cref{theta_ineq}) we have, due to our assumption,
	$$d\, \mu(q,v,\eta)\; \max\set{ d_\HS(q,q_\tau),  4\; d\sqrt{n}\; d_\HP((v,\eta),(v_\tau,\eta_\tau)) } \leq \varepsilon \leq \frac{1}{5}.$$
\cref{lipschitz3} implies that $\mu(q_\tau,v_\tau,\eta_\tau)\leq \Theta(\varepsilon)^{-1} \mu(q,v,\eta)$. Hence,
	\begin{equation*}
	\int_{0}^{\tau^*}\Norm{(\dot{v}_\tau,\dot{\eta}_\tau)} \d \tau\leq \frac{\mu(q,v,\eta)}{\Theta(\varepsilon)} \int_{0}^{\tau^*}  \Norm{\dot{q}_\tau} \d \tau
	<   \frac{\mu(q,v,\eta)}{\Theta(\varepsilon)} d_\HS(q,q'),
	\end{equation*}
contradicting our assumption.
\end{proof}
\begin{proof}[Proof of \cref{Theorem_4.4}]
The proof is adapted from \cite[p. 335--338]{condition}.

Let $E:=E_{f,g}=\cset{q_\tau}{0\leq \tau\leq 1}\subset \HS(\cH^n)$ be the geodesic path connecting $f$ and~$g$ and let $$L:=L_{f,g}=\cset{(q_\tau,v_\tau,\eta_\tau)}{0\leq \tau\leq 1}$$
be the lifting of $E$ at $(g,v,\eta)$. As in the proof of \cref{ALH-lemma} we argue that, if $E\cap\bSigma=\emptyset$, then $L\subset\hV$ (this means the lifted path does not contain the trivial solution $[0:1]$). Let
  $$K:=K(f,(g,v,\eta))$$
denote the number of iterations of algorithm EALH and let
	$$0=\tau_0 < \tau_1 <\ldots < \tau_K =1,\quad (w,\xi)=(w_0,\xi_0),\ldots,(w_K,\xi_K).$$
be the sequences of $\tau$'s and $(w,\xi)$'s generated by EALH, respectively. Further, for $0\leq i\leq K$ we denote $q_i:=q_{\tau_i}$ and $(v_i,\eta_i):=(v_{\tau_i},\eta_{\tau_i})$. We put
	\begin{equation*}
	\varepsilon:=0.04,\quad \Theta:=\Theta(\varepsilon)\approx 0.87488,\quad \chi:=2\Theta-1 \approx 0.74976 < 1.
	\end{equation*}
\begin{claim}
For all $i\in\set{0,\ldots,K}$ the following holds.
	\begin{equation}\label{ALH1.1}
	4d^2\sqrt{n} \mu(q_i,v_i,\eta_i)\, d_\HP((w_i,\xi_i),(v_i,\eta_i)) \leq \varepsilon
	\end{equation}
\end{claim}
\noindent\textit{Proof of the Claim} We will proceed by induction. By construction \cref{ALH1.1} holds for $i=0$. Let us assume that \cref{ALH1.1} holds for all $i\leq j$ for some fixed $j$. From \cref{lipschitz1} (2) we get that
	\begin{equation}\label{ALH1.11}
	(1-\varepsilon)^2 \mu(q_j,v_j,\eta_j)\leq \mu(q_j,w_j,\xi_j).
	\end{equation}
The stepsize of the algorithm EALH is defined such that
	$$d_\HS(q_j,q_{j+1}) \leq  \frac{\varepsilon(1-\varepsilon)^4  \chi \Theta}{4\; d^{2} \sqrt{n}\; \mu(q_j,w_j,\xi_j)^2}.$$
Using \cref{ALH1.11} we get $d_\HS(q_j,q_{j+1}) \leq \frac{\varepsilon \chi  \Theta}{4\; d^{2} \sqrt{n}\; \mu(q_j,v_j,\eta_j)^2}$ from this. In consequence, \cref{ALH-lemma} implies \begin{equation}\label{hallohallo}
d_\HP((v_j,\eta_j),(v_{j+1},\eta_{j+1})) \leq \frac{\varepsilon\chi}{4d^{2} \sqrt{n}\mu(q_j,v_j,\eta_j)}.
\end{equation}
Using that $2d\;\mu(q_j,v_j,\eta_j)\geq 1$, $d\geq 2$, $\Theta<1$ and $\chi<1$ we get
	\begin{equation*}
	d\,\mu(q_j,v_j,\eta_j) \, \max\set{ d_\HS(q_j,q_{j+1}),4\;d\sqrt{n}\; d_\HP((v_j,\eta_j),(v_{j+1},\eta_{j+1}))} \leq \varepsilon,
	\end{equation*}
so that \cref{lipschitz3} implies
	\begin{equation}\label{ALH1.5}
	 \Theta\; \mu(q_j,v_j,\eta_j)\leq\mu(q_{j+1},v_{j+1},\eta_{j+1}) \leq \frac{1}{\Theta}\; \mu(q_j,v_j,\eta_j).
	\end{equation}
We use the triangle inequality to get
	$$
	d_\HP((w_j,\xi_j),(v_{j+1},\eta_{j+1}))
	\leq \;d_\HP((w_j,\xi_j),(v_j,\eta_j))+d_\HP((v_j,\eta_j),(v_{j+1},\eta_{j+1})).$$
Applying \cref{hallohallo} to this inequality and our assumption, that \cref{ALH1.1} holds for $i$, yields
  \begin{align} 	d_\HP((w_j,\xi_j),(v_{j+1},\eta_{j+1}))\leq \;\frac{\varepsilon(1+\chi)}{4\;d^2\sqrt{n}\; \mu(q_j,v_j,\eta_j)}
  = \frac{\varepsilon \Theta}{2\;d^2\sqrt{n}\; \mu(q_j,v_j,\eta_j)},\label{ALH1.6}
	\end{align}
so that
	\begin{align*}
	d^2\sqrt{2n}\,  \mu(q_{j+1},v_{j+1},\eta_{j+1}) d_\HP((w_j,\xi_j),(v_{j+1},\eta_{j+1})&)
	 \leq \frac{\varepsilon \Theta}{\sqrt{2}}\, \frac{\mu(q_{j+1},v_{j+1},\eta_{j+1})}{\mu(q_j,v_j,\eta_j)}\stackrel{\cref{ALH1.5}}{\leq}  \frac{\varepsilon}{\sqrt{2}};
	\end{align*}
Together with \cref{Theorem_4.3} this implies $\gamma(q_{j+1},v_{j+1},\eta_{j+1}) < 0.1$. Using that $d\geq 2$, from \cref{2dmu} (2) that $2d\,\mu(q_j,v_j,\eta_j)\geq 1$ and $\Theta< 1$ we further get   from \cref{ALH1.6} that
$$d_\HP((w_j,\xi_j),(v_{j+1},\eta_{j+1})) \leq \frac{\varepsilon}{2} =0.02.$$
Recall from \cref{u_and_delta} the definitions of $u(r)$ and $\delta(r)$. If we put $r=0.999933...$, we have $u(r)> 0.1$ and $\delta(r)= 0.02$. Using \cref{def_approx_eigenpair} and \cref{gamma_thm} we see that $(w_j,\xi_j)$ is an approximate eigenpair of $q_{j+1}$ with associated eigenpair $(v_{j+1},\eta_{j+1})$. Applying Newton's method with respect to $q_{j+1}$ to $(w_j,\xi_j)$ halves the distance to $(v_{j+1},\eta_{j+1})$, so
	\begin{align*}
	d_\HP((w_{j+1},\xi_{j+1}),(v_{j+1},\eta_{j+1})) &\hspace{0.15cm}\stackrel{\text{\cref{gamma_thm}}}{\leq} \frac{1}{2} \; d_\HP((w_{j},\xi_{j}),(v_{j+1},\eta_{j+1}))\\
	 &\hspace{0.7cm}\stackrel{\cref{ALH1.6}}{\leq}\frac{\varepsilon\,\Theta }{4\;d^2\sqrt{n}\, \mu(q_j,v_j,\eta_j)}\\
	 &\hspace{0.7cm}\stackrel{\cref{ALH1.5}}{\leq}\frac{\varepsilon }{4\,d^2\sqrt{n}\, \mu(q_{j+1},v_{j+1},\eta_{j+1})}
	 \end{align*}
which is \cref{ALH1.1} for $i=j+1$. The correctness of EALH follows from this.

It remains to estimate $K$, the number of iterations.
In the same way we deduced \cref{ALH1.5}, we can deduce that for any $0\leq i\leq K-1$ and $\tau_i\leq\tau\leq \tau_{i+1}$ we have $\Theta\; \mu(q_i,v_i,\eta_i)\leq\mu(q_{\tau},v_{\tau},\eta_{\tau})$. Hence
	\begin{equation*}
	\int_{\tau_i}^{\tau_{i+1}} \mu(q_\tau,v_\tau,\eta_\tau)^2\d \tau \geq   \int_{\tau_i}^{\tau_{i+1}} \Theta^2 \mu(q_i,v_i,\eta_i)^2 \d \tau = \Theta^2 \;\mu(q_i,v_i,\eta_i)^2\; (\tau_{i+1}-\tau_i)
	\end{equation*}
As the stepsize we have put $$(\tau_{i+1}-\tau_i) = \frac{\chi \,\varepsilon (1-\varepsilon)^4 \Theta(\varepsilon)}{ 4\alpha d^{2} \sqrt{n}\, \mu(q_i,w_i,\xi_i)^2}, \quad \text{where } \alpha=d_\HS(f,g).$$
Furthermore, \cref{ALH1.1} and \cref{lipschitz1} (2) imply that for any $0\leq i\leq K$ we have $$\vartheta(\varepsilon)\; \mu(q_i,w_i,\xi_i)\leq  \mu(q_i,v_i,\eta_i),$$
where $\vartheta(\varepsilon):=1-(1-\varepsilon)^{-2}+\cos(\tfrac{\varepsilon}{4})$. Thus
	$$\int_{\tau_i}^{\tau_{i+1}} \mu(q_\tau,v_\tau,\eta_\tau)^2\d \tau \geq  \frac{\Theta^2(1-\varepsilon)^4\varepsilon\vartheta(\varepsilon)^2\chi}{ 4 \; d^{2} \sqrt{n}\;d_\HS(f,g)},$$
which implies
	$$\int_{0}^{1} \mu(q_\tau,v_\tau,\eta_\tau)^2\d \tau \geq   K\; \frac{\Theta^2(1-\varepsilon)^4\varepsilon\vartheta(\varepsilon)^2\chi}{ 4\;  d^{2}\sqrt{n}\;d_\HS(f,g)},$$
For $\varepsilon=0.04$ we have $$\frac{1}{4}\Theta^2(1-\varepsilon)^4\varepsilon\vartheta(\varepsilon)^2\chi \geq  \frac{1}{246}$$ and therefore
	$$K\leq 246 \;d^2\sqrt{n}\;d_\HS(f,g)\; \int_{0}^{1} \mu(q_\tau,v_\tau,\eta_\tau)^2\d \tau $$
	as claimed.
\end{proof}
\subsection{Average analysis of algorithm LVEALH}\label{sec:avg_anal}
In this section we prove \cref{Theorem_4.6}. The proof will be rather short as most of the work has already been done in previous sections. Recall from \cref{notation_C} the notation
\begin{equation*}\cC(f,(g,v,\eta)) =\int_0^1\mu(q_\tau,v_\tau,\eta_\tau)^2 \d \tau,
\end{equation*}
Let $$\cC(f):=\mean_{(g,v,\eta)\sim\rho} \cC(f,(g,v,\eta))$$ and define, as in \cref{part2.2:f_av_dfn},
 	$$(\mu^2)_{\mathrm{av}}(f) := \frac{1}{\cD(n,d)} \int_{(f,v,\eta)\in \pi^{-1}(f)} \mu(f,v,\eta)^2\, \d(f,v,\eta);$$
Since the density $\rho$ is defined such that $g\sim N(\cH^n)$ and $(v,\eta)$ is chosen uniformly at random from the $\cD(n,d)$ many h-eigenpairs of $g$, we have
	\begin{equation*}
	\mean\limits_{f \sim\text{Unif}\;\HS(\cH^n)} \cC(f) =\mean\limits_{f \sim\text{Unif}\;\HS(\cH^n)}  \mean\limits_{g\sim N(\cH^n)} \int_{q\in E} \mu^2_{\mathrm{\;av}}(q)\,\d q,
	\end{equation*}
where $E\subset \HS(\cH^n)$ denotes the geodesic path between $f$ and $\Norm{g}^{-1}g$. Clearly, the integral is independent of the norm of $g$, so we may as well take the expectation over $g\sim\mathrm{Unif}\;\HS(\polspace)$. Furthermore, \cref{condition_for_eigenpairs} (3) implies that for all $q\in\cH^n$ we have $\mu^2_{\mathrm{\;av}}(q) \leq \hatmu^2_{\mathrm{\;av}}(q) $ and hence
 \begin{equation}\label{r6}
	\mean\limits_{f \sim\text{Unif}\;\HS(\cH^n)} \cC(f) \leq \mean\limits_{f,g \sim\text{Unif}\;\HS(\cH^n)}  \int_{q\in E} \hatmu^2_{\mathrm{\;av}}(q)\,\d q.
	\end{equation}
 The following is \cite[Eq. (4.3)]{Beltran2011} specialized to our scenario.
\begin{lemma}\label{beltran_pardo_lemma}
Let $\HS:=\HS(\polspace)$. For any measurable function $\phi:\HS\to\HR$ we have
$$\mean\limits_{f,g \sim\mathrm{Unif}(\HS)}\; \int_{h\in E} \phi(h)\, \d q  = \frac{\pi}{2} \mean\limits_{h\sim\mathrm{Unif}(\HS)} \;\phi(h),$$
where $E\subset\HS(\cH^n)$ denotes the geodesic path between $f$ and $g$.
\end{lemma}
Combining \cref{r6} with \cref{beltran_pardo_lemma} we get that
	\begin{equation*}
	\mean_{f \sim\mathrm{Unif}\;\HS(\cH^n)} \cC(f) \leq \frac{\pi}{2}\mean_{q \sim\text{Unif}\;\HS(\cH^n)}   \hatmu^2_{\mathrm{\;av}}(q)\leq \frac{40\pi\,nN}{d},
\end{equation*}
the last inequality by \cref{mu_expectation0}. Thus, \cref{Theorem_4.6} is proved.

\subsection{Analysis of the sampling method}
\label{sec:Theorem_4.7}
In this section we prove \cref{Theorem_4.7}. We show that Draw-from-$\rho^\star$ on the average needs $\cO(n^3+dnN)$ arithmetic operations (including drawing from $N(\polspace$)), where $N:=\dim_\HC(\polspace)$. "Average", because the success of the \textbf{if}-clause in line~9 of the algorithm is a random variable.

Step 3. of algorithm Draw-from-$\rho^\star$  can be done via Beltran-Pardo randomization \cite[Section~17.6]{condition}. By \cite[Proposition 17.21]{condition} this randomization can be implemented with $\cO(N)$ draws from $N(\HC)$ and $\cO(dnN)$ arithmetic operations.

For Step $5.$ draw $B\sim N(\HC^{(n-1)\times(n-1)})$. Let $Q$ be the $Q$-factor in the QR-decomposition of $B$. Then we put $U'=\mathrm{diag}(r_{ii}/\norm{r_{ii}})\,Q$, where the $r_{ii}$ are the diagonal elements of the $R$ factor in the QR decomposition of $B$. This algorithm yields $U'\sim \cU(n-1)$ uniformly at random (see, e.g., \cite[Sec. 10.1]{Armentano2015a}). Compute $U''\in\cU(n)$ with $U''e_1=v$, for instance via Gram-Schmidt algorithm. Putting $$U=\begin{bmatrix} v, & U'' \begin{bmatrix} 0\\U'\end{bmatrix}\end{bmatrix}$$ we have $U\in\cU(n)$, such that $Ue_1=v$, uniformly at random. The dominating complexity is the $QR$ decomposition, which requires $\cO(n^3)$ arithmetic operations. Summarizing we can implement Step $5.$ with  $\cO(n^3)$ arithmetic operations.

Given $v\in\HP(\HC^n)$ we can draw $a \sim N(v^\perp)$ by drawing $a'\sim N(\HC^{n-1})$ and then put $a=U(0,a')^T$, where $U$ is as above. This requires $\cO(n^2)$ arithmetic operations. An implementation for drawing $h\sim N(R(v))$ is  given in \cite[Algorithm 17.7]{condition}. Its number of arithmetic operations is $\cO(N)$. Hence, step 4. can be implemented using $\cO(N)$ operations.

For drawing $r$ with density $e^{-r}\;\mathbf{1}_{\set{r\geq 0}}(r)$, one can draw $z\sim N(\HC)$ and then put~$r=\Norm{z}^2$.

Altogether this shows that steps $3.$-$7$. require $\cO(n^3+dnN)$ arithmetic operations.

The number of times we have to execute steps $9.$-$14$. is imposed by the {\bf if}-clause in step~9. Let $E$ denote the expected number of iterations  and let \begin{equation}\label{C_from_if_statement}C:=\Prob\limits_{(\Hf,v),\eta}\set{\Norm{\restr{\deriv{\Hf}{v}}{v^\perp}+\eta^{d-1}\restr{I}{v^\perp}}_F \leq \Norm{\restr{\deriv{\Hf}{v}}{v^\perp}}_F}\end{equation}
be the probability of the {\bf if}-clause being true. For any $k\geq 1$ we have
	$$\Prob\set{ \text{The {\bf if}-statement is first true in the $k$-th iteration}} = (1-C)^{k-1}C,$$
which implies $E= \sum_{k=1}^\infty k (1-C)^{k-1}C = C^{-1}$. In \cref{C} below we show $C\geq (5\sqrt{\pi}\;n )^{-1}$. Observe that every iteration requires $\cO(1)$ many draws from~$N(\HC^n)$ and the computation of a derivative and its Frobenius norm, which can be done with $\cO(N)$ many arithmetic operations.

Summarizing, our implemention of algorithm Draw-from-$\rho^\star$ requires an expected number of $\cO(n^3+dnN)$ arithmetic operations.
\subsubsection{The probability of the if-clause being true.}\label{sec:rho_star}
Let $\Omega\subset \cH^{n-1}\times \HP(\HC^n)\times\HC\times \cU(n)\times \HC^n \times
 \cH$ be the probability space
\begin{align*}
	\Omega :=& \Big\{(\Hf,v,\eta,U,a,h)   \mid
	\Hf(v)=0,\; Ue_1=v,\;a\in v^\perp,\; h\in R(v)\Big\},
	\end{align*}
where the random variables $\Hf,v,\eta,U,a,h$ have the distribution that is induced by the algorith Draw-from-$\rho^\star$; that is,
$$(\Hf,v)\sim \rho_{\mathrm{BP}}, U\sim \mathrm{Unif} \cset{U\in\cU(n)}{Ue_1=v}, a\sim N(v^\perp), h\sim N(R(v)), \eta\sim\beta.$$
Here $\rho_{\mathrm{BP}}$ is the distribution from Beltr\'an Pardon randomization \cite{Beltran2011} and $\beta$ is the distribution of~$\eta$ as described in \cref{beta_density} below. Consider the subset
	\begin{equation*}
	\Omega^* = \cset{(\Hf,v,\eta,U,a,h)\in \Omega}{\Norm{\restr{\deriv{\Hf}{v}}{v^\perp} + \eta^{d-1}\restr{I}{v^\perp}}_F \leq \Norm{\restr{\deriv{\Hf}{v}}{v^\perp}}_F} \label{omega_star}
\end{equation*}
with $\restr{I}{v^\perp}$ being restriction of the identity to $v^\perp$. The $C$ from \cref{C_from_if_statement} is then given as $C=\Prob(\Omega^*).$ The plan is now to compute first the density of $\eta$ and to compute $\Prob(\Omega^*)$ afterwards.
\begin{lemma}\label{beta_density}
If we choose $r$ with density $e^{-r}\;\mathbf{1}_{\set{r\geq 0}}(r)$ and $\phi\in[0,2\pi)$ uniformly at random, putting $\eta:=r^\frac{1}{2(d-1)} \exp(i\phi)$ defines a random variable with density
	$$\beta(\eta):= \frac{(d-1)}{\pi}\;\norm{\eta}^{2(d-2)}\;\exp(-\norm{\eta}^{2(d-1)}).$$
\end{lemma}
\begin{proof}
Put $s:=r^\frac{1}{2(d-1)}$, such that we have $2(d-1)s^{2(d-2)+1}\d s = \d r$ and
$$\exp(-r)\;\mathbf{1}_{\set{r\geq 0}}(r) \d r = 2(d-1)s^{2(d-2)+1}\;\exp(-s^{2(d-1)}) \;\mathbf{1}_{\set{s\geq 0}}(s) \d s.$$
We have defined the argument of the random variable $\eta$ as being uniform distributed in $[0,2\pi)$. Changing from polar to euclidean coordinates $\eta := s \exp(i\phi)$ we get the density of $\eta$ given by $$\frac{1}{\pi}\;(d-1)\norm{\eta}^{2(d-2)}\;\exp(-\norm{\eta}^{2(d-1)}) \;\d \eta,$$
which shows the assertion.
\end{proof}
We can now prove the asserted bound on $C$.
\begin{lemma}\label{C} We have $$C=\Prob(\Omega^*)\geq  \frac{1}{5\sqrt{\pi}\,n}.$$
  \end{lemma}
  \begin{proof}
In the proof of \cite[Lemma 17.18]{condition} it is shown that the density $\rho_\mathrm{PB}$ is invariant under unitary transformations. Moreover, by \cite[Lemma 17.18 for linear systems]{condition}  and \cite[Lemma~17.19]{condition} the push-forward distribution of $\rho_\mathrm{PB}$ under the map $(\Hf,v)\mapsto \sqrt{d}^{\;-1}\deriv{\Hf}{v} \in\HC^{(n-1)\times n}$ is the complex standard normal distribution on $\HC^{(n-1)\times n}$. Writing
  $A:=\eta^{d-1}\restr{I}{e_1^\perp}.$
we have that
\begin{align*}
\Prob(\Omega^*)
&=\mean\limits_{\eta}\;\Prob\limits_{(\Hf,v)}\set{\Norm{\restr{\deriv{\Hf}{e_1}}{e_1^\perp} + A}_F \leq \Norm{\restr{\deriv{\Hf}{e_1}}{e_1^\perp}}_F}\\
&=\mean\limits_{\eta}\;\Prob\limits_{Z \sim N(\HC^{(n-1)\times (n-1)})}\set{\Norm{Z - \sqrt{d}^{\,-1}A}_F \leq \Norm{Z}_F},
\end{align*}
Note that $\Norm{A}_F\geq \lVert\sqrt{d}^{\,-1} A\rVert_F.$
From \cref{norm_inequality_gaussian} (1) get
\begin{align*}
\Prob(\Omega^*)\geq \mean\limits_{\eta}\;\Prob\limits_{Z}\set{\Norm{Z +A}_F \leq \Norm{Z}_F}.
\end{align*}
Then, using that $\Norm{A}_F= \sqrt{n-1}\, \lvert\eta\rvert^{d-1}$ and \cref{norm_inequality_gaussian} (2) we get
  	\begin{align*}
  	\Prob(\Omega^*)\geq \mean\limits_{\eta}\; \frac{1}{\sqrt{\pi}} \frac{2\exp\big(\frac{-(n-1)\norm{\eta}^{2(d-1)}}{2}\big)}{\sqrt{n-1}\norm{\eta}^{d-1} + \sqrt{(n-1)\norm{\eta}^{2(d-1)}+8}}.
  \end{align*}
Put $r:=\norm{\eta}^{2(d-1)}$. By definition $r$ has density $e^{-r}\;\mathbf{1}_{\set{r\geq 0}}(r)$. This shows
  \begin{align*}
  	\Prob(\Omega^*)\geq  \frac{2}{\sqrt{\pi}}\int_{r\geq 0} \frac{1}{\sqrt{(n-1)r} + \sqrt{(n-1)r+8}} \;\exp\big(-r\big(\tfrac{(n-1)}{2}+1\big)\big) \d r.
  	\end{align*}
Making a change of variables $t:=(n-1)r$ the latter expression becomes
  	\begin{align*}
  	 &\frac{2}{\sqrt{\pi}\,(n-1)}\int_{t\geq 0} \frac{1}{\sqrt{t} + \sqrt{t+8}} \;\exp\big(-t(\tfrac{1}{2}+\tfrac{1}{n-1})\big)\; \d t
  	 \end{align*}
Using $n\geq 2$ we have $\tfrac{1}{2}+\tfrac{1}{n-1}\leq \tfrac{3}{2}$, so that
$$	\Prob(\Omega^*)\geq \frac{2}{\sqrt{\pi}(n-1)}\int_{t\geq 0} \frac{1}{\sqrt{t} + \sqrt{t+8}} \;\exp\big(-\tfrac{3}{2}t\big)\; \d t.$$
A computer based calculation shows $\int_{t\geq 0} \frac{1}{\sqrt{t} + \sqrt{t+8}} \;\exp\left(-\frac{3}{2} t\right)\; \d t \geq \frac{1}{10}$. This finishes the proof.
  \end{proof}
\subsection{Average analysis of algorithm LVEALHWS}
In this section we prove \cref{Theorem_4.8}. To this end, let
\begin{equation}\label{phi_for_prop}
\phi: \hV\to \HR,\; (g,v,\eta)\mapsto \mean\limits_{f\sim\mathrm{Unif}\;\HS(\polspace)}\,\cC \,\left(f,(\Norm{g}^{-1} g,v,\Norm{g}^\frac{-1}{d-1}\eta)\right).
\end{equation}
By \cref{2dmu} (4), $\phi$ is unitarily invariant, that is,  $\phi(g,v,\eta)=\phi(U.(g,v,\eta)).$ for all $U\in\cU(n)$. Further, by definition $\phi$ is scale invariant: for all $s\in\HC\backslash\set{0}$ we have $\phi(g,v,\eta)=\phi(s^{d-1}g,v,s\eta)$.
\begin{proof}[Proof of \cref{Theorem_4.8}]
Apply \cref{hallo1000} below to the function $\phi$ from \cref{phi_for_prop}.
\end{proof}
The proof of the following proposition is adapted from \cite[Section 10.2]{Armentano2015a}.
\begin{prop}\label{hallo1000}
Let $\Xi: \hV \to \HR$ be a measurable, unitarily invariant, scale invariant nonnegative function. Then $$\mean\limits_{(f,v,\eta)\sim\rho^\star} \Xi(f,v,\eta) \leq 10\sqrt{\pi}\,n\,\mean\limits_{(f,v,\eta)\sim\rho} \Xi(f,v,\eta).$$
\end{prop}
\begin{proof}
While the proof itself is very technical, its underlying idea is explained quickly: Recall from \cref{omega_star} the definition of $\Omega^*$. The idea is to show that on $\Omega^*$ we have the inequality $\rho^\star\leq \frac{\rho}{\Prob(\Omega^*)}$, so that we can have the following bound for a random variable $X$ on $\Omega$:
	$$\mean_{X\sim \rho^\star}[X] \leq \int_{\Omega^*}  X\,\frac{\rho}{\Prob(\Omega^*)}\; \d \Omega^* \leq \frac{1}{\Prob(\Omega^*)}\int_{\Omega}  X\,\rho\; \d \Omega=\frac{\mean\limits_{X\sim\rho}[X]}{\Prob(\Omega^*)}. $$
\cref{hallo1000} is then implied using \cref{C}. Denote the characteristic function of~$\Omega^*$ by
$$\mathbf{1}_{\Omega^*}:=\mathbf{1}_{\Omega^*}(\Hf,v,\eta).$$
As before we put~$C:=\Prob(\Omega^*)$. By construction, the density~$\rho^\star$ is the density of the \emph{conditional distribution on $\Omega^*$}, associated to the probability measure ${\Prob}^*(Y) = C^{-1}\, \mean_{(\Hf,v),a,h,U,\eta}\; \mathbf{1}_Y$ for measurable $Y\subset \Omega^*$.
We therefore have
	\begin{equation}\label{rhostar1.1a}
	\mean\limits_{(f,v,\eta)\sim\rho^\star} \Xi(f,v,\eta) = \frac{1}{C} \mean\limits_{(\Hf,v),a,h,U,\eta} \Xi(f,v,\eta)\;\mathbf{1}_{\Omega^*},
	\end{equation}
where (recall that $R(v):=\cset{h\in\cH_{n,d}}{h(v)=0,\deriv{h}{v}=0}$)
$$(\Hf,v)\sim \rho_{\mathrm{BP}}, U\sim \mathrm{Unif} \cset{U\in\cU(n)}{Ue_1=v}, a\sim N(v^\perp), h\sim N(R(v)), \eta\sim\beta.$$
By assumption $\Xi$ is unitarily invariant and it is easily seen that~$\mathbf{1}_{\Omega^*}$ is unitarily invariant, too. The function $\Xi$ is nonnegative and therefore by Tonelli's theorem the expectation in \cref{rhostar1.1a} is independent of the order in which we integrate, so that
\begin{equation}\label{rhostar1.1}
\mean\limits_{(f,v,\eta)\sim\rho^\star} \Xi(f,v,\eta) = \frac{1}{C} \mean\limits_{U,a,h,\eta}\;\mean\limits_{(\Hf,v)\sim \rho_{\mathrm{PB}}} \Xi(f,v,\eta)\;\mathbf{1}_{\Omega^*},
\end{equation}
For a moment we consider the inner expectation. Let $\cZ(v):=\cset{B\in\HC^{(n-1)\times n}}{Bv =0}$
and put $B:=\sqrt{d}^{\, -1}\deriv{\Hf}{v}$. Note that, since $\Hf(v)=0$, we have $B\in \cZ(v)$. Decomposing $\cH^{n-1}$ orthogonally as in the preliminary section \cref{sec:Preliminaries}, we see that $\Hf$ is determined by $B,v$ and $k\in R(v)^{n-1}$. By \cite[Algorithm 17.6]{condition} we have $k\sim N(R(v)^{n-1})$. Moreover, by \cite[Lemma 17.19]{condition} the pair $(B,v)$ follows the so called \emph{standard distribution} on the linear solution manifold $$W=\cset{(B,v)\in\HC^{(n-1)\times n} \times \HP(\HC^n)}{B\in\cZ(v)}.$$
The equation before \cite[(17.20)]{condition} says that the standard distribution  has the density
  $$\rho_{\mathrm{st}}(B,v)= \varphi(B) \,\mathrm{NJ}p_1(B,v),$$
where $\varphi(B)$ is the density of $N(\HC^{(n-1)\times n})$, $p_1$ is the projection onto the first factor and $\mathrm{NJ}p_1(B,v)$ is the \emph{normal jacobian} of $p_1$ at $(B,v)$; see, e.g.,  \cite[Section 17.3]{condition}.

By the foregoing, for fixed $U,a,h,\eta$ we have
$$\mean\limits_{(\Hf,v)\sim \rho_{\mathrm{BP}}} \Xi(f,v,\eta)\;\mathbf{1}_{\Omega^*} = \int_{(B,v)\in W}  \mean_{k}\;\Xi(f,v,\eta)\,\mathbf{1}_{\Omega^*} \, \varphi(B)\,\mathrm{NJ}p_1(B,v)\, \d(B,v),$$
where $k\sim N(R(v)^{n-1})$. We use the coarea formula (e.g. \cite[Theorem 17.8]{condition}) on the projection $p_2:W\to \HP(\HC^n)$ to deduce that $\mean_{(\Hf,v)} \Xi(f,v,\eta)\;\mathbf{1}_{\Omega^*}$ is equal to
\begin{align*}
  \int_{v\in \HP(\HC^n)} \bigg(\int_{(B,v) \in p_2^{-1}(v)}  \; \mean_{k}\; \Xi(f,v,\eta)\,\mathbf{1}_{\Omega^*}\, \varphi(B)\,\frac{\mathrm{NJ}p_1(B,v)}{\mathrm{NJ}p_2(B,v)}\,\d p_2^{-1}(v) \bigg) \d v.
\end{align*}
We use the characterization from \cite[Lemma 17.14]{condition}:
$\mathrm{NJ}p_1(B,v)= \det(BB^*)\,\mathrm{NJ}p_2(B,v).$ Note that $(B,v)\in p_2^{-1}(v)$ is equivalent to $B\in \cZ(v)$, and that for all matrices $B \in \cZ(v)$ we have $\det(BB^*)=\det (\restr{B}{v^\perp})^2$. Interchanging integration over $k$ and $B$ we see that
	\begin{align*}
 \mean_{(\Hf,v)} \Xi(f,v,\eta)\;\mathbf{1}_{\Omega^*} = \int_{v\in\HS(\HC^n)}\; \mean_{k}\; \Xi(f,v,\eta)\,\mathbf{1}_{\Omega^*} \Big(\int_{B\in\cZ(v)}\,\det ( \restr{B}{v^\perp})^2 \varphi(B)\d B\Big)\,\d v.
	\end{align*}
Taking expectation over all random variables, by \cref{rhostar1.1a}, we therefore have
\begin{align*}
&\mean\limits_{(f,v,\eta)\sim\rho^\star} \Xi(f,v,\eta)
=  \frac{1}{C}\mean_{U,a,h,\eta}  \int_{v\in\HS(\HC^n)}\; \mean_{k}\; \Xi(f,v,\eta)\,\mathbf{1}_{\Omega^*}\; \Big(\int_{B\in\cZ(v)}\,\det ( \restr{B}{v^\perp})^2 \varphi(B)\d B\Big)\,\d v.
\end{align*}
By Tonelli's theore the latter is equal to
\begin{align}\label{hallo5}
\frac{1}{C}\int_{v\in\HS(\HC^n)}\;\mean_{U,a,h,\eta,k}  \Xi(f,v,\eta)\,\mathbf{1}_{\Omega^*}\; \Big(\int_{B\in\cZ(v)}\,\det ( \restr{B}{v^\perp})^2 \varphi(B)\d B\Big)\,\d v.
\end{align}
To keep track, recall that we have
  $$f= U\begin{bmatrix}\Hf_0\\ \Hf\end{bmatrix}, \text{ with } \begin{bmatrix}\Hf_0\\ \Hf\end{bmatrix}=\eta^{d-1}X_1^de_1 + X_1^{d-1}  \sqrt{d}\begin{bmatrix} a^T\\ B \end{bmatrix} (X_2,\ldots,X_n)^T + \begin{bmatrix} h\\ k \end{bmatrix}.$$
By unitary invariance, the integrand in \cref{hallo5} is invariant of $v$ and $U$. In that integral we may replace $v$ by $e_1$ and integrate over $U$ and $v$ so that
\begin{align*}
&\mean\limits_{(f,v,\eta)\sim\rho^\star} \Xi(f,v,\eta)
=  \frac{\mathrm{vol}\, \HP(\HC^n)}{C}\mean_{a,h,\eta,k}  \Xi(f,e_1,\eta)\,\mathbf{1}_{\Omega^*}\; \Big(\int_{B\in\cZ(e_1)}\,\det ( \restr{B}{v^\perp})^2 \varphi(B)\d B\Big),
\end{align*}
where the random variables are $a\sim N(\set{0}\times\HC^{n-1}),\; h\sim R(e_1),\; k\sim R(e_1)^{n-1}$ and $\eta\sim \beta.$ Observe that $\cZ(e_1) = \cset{B\in \HC^{(n-1)\times n}}{Be_1 = 0}$ is equal to $\cset{[0,M]}{M\in\HC^{(n-1)\times (n-1)}}$ and that $\restr{[0,M]}{e_1^\perp}=M$.  Moreover, note that on $\cZ(e_1)$ we have $$\varphi([0,M])=\varphi_{\HC^{n-1}}(0)\varphi_{\HC^{(n-1)\times (n-1)}}(M)=\pi^{-(n-1)} \varphi_{\HC^{(n-1)\times (n-1)}}(M).$$
Altogether, this shows that
$$\int_{B\in\cZ(e_1)}\,\det ( \restr{B}{v^\perp})^2 \varphi(B)\d B = \frac{1}{\pi^{n-1}}\,\mean\limits_{M\sim N(\HC^{(n-1)\times (n-1)})} \det(M)^2,$$
so that
\begin{equation}
	\mean\limits_{(f,v,\eta)\sim\rho^\star} \Xi(f,v,\eta)=\frac{\mathrm{vol}\,\HP(\HC^n)}{\pi^{n-1}\;C}\;\mean\limits_{a,h,\eta,M,k} (\det M)^2\; \Xi(f,e_1,\eta)\mathbf{1}_{\Omega^*},\label{rhostar3}
	\end{equation}
where
$k\sim N(R(e_1)^{n-1}), M\sim N(\HC^{(n-1)\times (n-1)}), a\sim N(\HC^{n-1}), h\sim N(R(e_1))$ and $\eta\sim \beta$.
Let us fix all the random variables but $M$ now, and write the expectation over $M$ as the integral
\begin{equation*}\mean\limits_{M} \;(\det M)^2\; \Xi(f,e_1,\eta)\mathbf{1}_{\Omega^*}=\frac{1}{\pi^{(n-1)^2}}\int_M (\det M)^2\; \Xi(f,e_1,\eta)\mathbf{1}_{\Omega^*} \varphi(M) \d M,\end{equation*}
where $\varphi$ is the density of $N(\HC^{(n-1)\times (n-1)})$. Recall that $[0,M]=\sqrt{d}^{-1}\deriv{\Hf}{e_1}$ and that
$$\mathbf{1}_{\Omega^*}(\Hf,e_1,\eta)=1,  \text{ if and only if } \Norm{\restr{\deriv{\Hf}{e_1}}{e_1^\perp}+ \eta^{d-1}\restr{I}{e_1^\perp}} \leq \Norm{\restr{\deriv{\Hf}{e_1}}{e_1^\perp}}.$$
Note that $\restr{I}{e_1^\perp}$ is the $(n-1)\times(n-1)$ identity matrix $I_{(n-1)\times (n-1)}$, which---abusing notation---we also abbreviate by $I$. Setting $A:= M + \sqrt{d}^{-1}\eta^{d-1}I$ we get $\Norm{A}_F^2 = \Vert M+\sqrt{d}^{\;-1}\;\eta^{d-1} I \Vert_F^2$ and
\begin{align}\nonumber
\mathbf{1}_{\Omega^*}(M,\eta)=1&\stackrel{\text{by definition}}{\Longleftrightarrow}  \Norm{M+\sqrt{d}^{\;-1}\;\eta^{d-1} I }_F^2 \leq \Norm{M}_F^2\\
&\hspace{0.4cm}\Longleftrightarrow  \quad \exp(-\Norm{M}_F^2) \leq \exp(-\Norm{A}_F^2).\label{important_inequality}
\end{align}
Using also that $\mathbf{1}_{\Omega^*}(M,\eta)\leq 1$, we get
	\begin{align}
	&\mean\limits_{M}(\det M)^2 \Xi(f,e_1,\eta)\mathbf{1}_{\Omega^*}
	\leq \mean\limits_{A}\;(\det (A-\sqrt{d}^{-1}\eta^{d-1}I))^2 \Xi(f,e_1,\eta),\label{important_inequality2}
\end{align}
where $A\sim N(\HC^{(n-1)\times (n-1)})$. Plugging this into \cref{rhostar3} yields
\begin{align}\nonumber
	\mean\limits_{(f,v,\eta)\sim\rho^\star} \Xi(f,v,\eta)
  \leq\;&\frac{\mathrm{vol} \,\HP(\HC^n)}{\pi^{n-1}C}\; \mean\limits_{a,h,\eta,A}\;(\det (A-\sqrt{d}^{-1}\eta^{d-1}I))^2\; \Xi(f,e_1,\eta)\\
  \leq\;&\frac{\mathrm{vol} \,\HS(\HC^n)}{2\pi^{n}d^{n-1}\,C}\; \mean\limits_{a,h,\eta,k,A}\;(\det (\sqrt{d}A-\eta^{d-1}I))^2\; \Xi(f,e_1,\eta),\label{rhostar5}
\end{align}
where we have used $2\pi\,\mathrm{vol}\, \HP(\HC^n)=\mathrm{vol}\,\HS(\HC^n).$
We use \cref{beta_density} to write the integral over $\eta$ explicitly, so that \cref{rhostar5} becomes
  \begin{align}
    &\mean\limits_{(f,v,\eta)\sim\rho^\star} \Xi(f,v,\eta)\leq\frac{\mathrm{vol} \,\HS(\HC^n)}{2\pi^{n}d^{n-1}C} \;\int_{\eta\in\HC} \widehat{E}(\eta) \; \frac{(d-1)}{\pi}\; \norm{\eta}^{2(d-2)}\; e^{-\norm{\eta}^{2(d-1)}} \d \eta;\label{xxxxxxxxx}
\end{align}
where $$\widehat{E}(\eta):=\mean\limits_{a,h,k,A} \;(\det (\sqrt{d}A-\eta^{d-1}I))^2\; \Xi(f,e_1,\eta)$$
with
  $$A \sim N(\HC^{(n-1)\times (n-1)}),\; a\sim N(\HC^{n-1}),\; \begin{bmatrix} h\\ k \end{bmatrix}\sim N(R(e_1)^{n}).$$
Denote by $\widehat{\Xi}:\hatV \to \HR$ the map that that satisfies $\widehat{\Xi} = \Xi \circ \Pi$, where $\Pi:\hatV \to \hV$ is the canonical projection.
By \cref{part2.2:integration} we can identify
\begin{equation}
\frac{\mathrm{vol}\;\HS(\HC^n)}{2\pi^{n+1}\,\cD(n,d)} \int_{\eta\in\HC}(d-1)^2\widehat{E}(\eta) \;  \norm{\eta}^{2(d-2)}\; e^{-\norm{\eta}^{2(d-1)}} \d \eta=\mean\limits_{f\sim N(\polspace)}\widehat{\Xi}_\mathrm{av}(f)\label{hallo12345}
\end{equation}
where, as in \cref{part2.2:f_av_dfn2},
  \begin{equation}\label{myspaceSintegral} \widehat{\Xi}_\mathrm{av}(f)=\frac{1}{2\pi\cD(n,d)}\; \int_{(f,v,\eta)\in \hatpi^{-1}(f)} \widehat{\Xi}(f,v,\eta) \;\d(f,v,\eta).\end{equation}
If we replace in \cref{myspaceSintegral} the integration  over $\myspaceS$ by integration over $\cP$ we get
 $$\widehat{\Xi}_\mathrm{av}(f)=\frac{1}{\cD(n,d)}\; \int_{(f,v,\eta)\in \pi^{-1}(f)} \Xi(f,v,\eta) \;\d(f,v,\eta) ,$$
In \cref{part2.2:f_av_dfn} we denoted the last integral by $ \Xi_\mathrm{av}(f)$.
This we plug into \cref{hallo12345} and then apply the resulting equation to \cref{xxxxxxxxx} and get
$$\mean\limits_{(f,v,\eta)\sim\rho^\star} \Xi(f,v,\eta)\leq \frac{\cD(n,d)}{d^{n-1}(d-1)C} \mean\limits_{f\sim N(\polspace)}\Xi_\mathrm{av}(f).$$
Note that, since $d\geq 2$, we have
$$\frac{\cD(n,d)}{d^{n-1}(d-1)} = \frac{d^n-1}{d^{n-1}(d-1)} =\sum_{i=0}^{n-1}\frac{1}{d^i} \leq \sum_{i=0}^{\infty}\frac{1}{2^i} = 2 .$$
Moreover, by \cref{C}, $C\geq (5\sqrt{\pi}\,n)^{-1}$. Since $\Xi(f,v,\eta)$ is invariant under scaling of $f$, also $\Xi_\mathrm{av}$ is scale invariant. This implies
  $\mean_{f\sim N(\cH^n)}\Xi_\mathrm{av}(f)=\mean_{f\sim \mathrm{Unif}\,(\cH^n)}\Xi_\mathrm{av}(f)$; see, e.g.,  \cite[Corollary~2.23]{condition}.
By \cref{forward_distribution}  we have $\mean_{f\sim \mathrm{Unif}\,(\cH^n)}\Xi_\mathrm{av}(f) =  \mean_{(f,v,\eta)\sim \rho} \Xi(f,v,\eta)$, which finally yields
$$\mean\limits_{(f,v,\eta)\sim\rho^\star} \Xi(f,v,\eta)\leq 10\sqrt{\pi}\, n \;\mean\limits_{(f,v,\eta)\sim \rho} \Xi(f,v,\eta).$$
This finishes the proof.
\end{proof}
\section{Proof of Proposition \ref{Proposition_1.4}}\label{sec_proof_ill_posed}
By \cite[Proposition 16.10]{condition} the condition number of polynomial equation solving is given by
  $$\mu(\cF_f,(v,\eta))= \Norm{\cF_f}\,\Norm{(v,\eta)}^{d-1}\,\Norm{\big(\restr{\deriv{\cF_f}{(v,\eta)}}{(v,\eta)^\perp}\big)^{-1}}.$$
In \cref{2dmu} (1) we have shown that
   $$\Norm{\big(\restr{\deriv{\cF_f}{(v,\eta)}}{z_t^\perp}\big)^{-1}} \geq \frac{1}{d\,\left(\Norm{f}+\max\set{\norm{\eta}^{d-1},\norm{\eta}^{d-2}}\right)}$$
  so that
  \begin{equation}\label{part2.2:steps_ALH}
  \mu(\cF_f,(v,\eta))\geq  \frac{\Norm{\cF_f}\,\Norm{(v,\eta)}^{d-1}}{d\,\left(\Norm{f}+\max\set{\norm{\eta}^{d-1},\norm{\eta}^{d-2}}\right)}.
  \end{equation}
The norm is $\cF_f$ is given by
   $\Norm{\cF_{f}}^2=\Norm{f(X)-\ell^{d-1} X}^2 = \Norm{f}^2+\Norm{\ell^{d-1}X}^2 = \Norm{f}^2 + \frac{n}{d}.$
Moreover, since $f(v)=\eta^{d-1} v$ we have
  \begin{equation}\label{hallo8}\Norm{\eta}^{2(d-1)}\Norm{v}^2 = \Norm{f(v)}^2 \leq \Norm{f}^2 \Norm{v}^{2d};\end{equation}
the inequality by \cite[Lemma 16.5]{condition}. Because $\Norm{v}=1$, we have $\Norm{\eta}^{d-1}\leq \Norm{f}$. This implies
   \begin{equation}\label{part2.2:steps_ALH8}
   \frac{\Norm{f}^2 + \frac{n}{d}}{\left(\Norm{f}+\max\set{\norm{\eta}^{d-1},\norm{\eta}^{d-2}}\right)^2} \geq \frac{\Norm{f}^2}{\left(\Norm{f}+\max\set{\norm{\eta}^{d-1},\norm{\eta}^{d-2}}\right)^2}\geq \frac{1}{4}
   \end{equation}
and hence
$$\frac{\Norm{\cF_f}}{d\,\left(\Norm{f}+\max\set{\norm{\eta}^{d-1},\norm{\eta}^{d-2}}\right)}\geq \frac{1}{2d}.$$
Plugging into \cref{part2.2:steps_ALH} yields
\begin{equation}\label{hallo77}
   \mu(\cF_f,(v,\eta))\geq  \frac{\Norm{(v,\eta)}^{d-1}}{2d}.
 \end{equation}
By assumption, $(v,\eta)$ is chosen uniformly at random from the $\cD(n,d)$ many h-eigenpairs of~$f$. This implies that $\lambda:=\eta^{d-1}$ is an eigenvalue of~$f$ that is chosen uniformly from all of its eigenpairs. Theorem 1.1 from \cite{distr} implies that
$$\Prob\set{\norm{\lambda} >1}
  = \frac{d}{d^n-1} \sum_{k=1}^n d^{n-k} \Prob\limits_{X\sim \chi^2_{2k}}\set{X>2}.$$
From this we deduce
\begin{align*}\Prob\set{\norm{\eta}^2 > 1}&=\Prob\set{\norm{\eta}^{2(d-1)} > 1}
  =\Prob\set{\norm{\lambda}^2 >1}
  = \frac{d}{d^n-1} \sum_{k=1}^n d^{n-k} \Prob\limits_{X\sim \chi^2_{2k}}\set{X>2},
\end{align*}
Using $\Prob_{X\sim \chi^2_{2k}}\set{ X>2} \geq \Prob_{X\sim \chi^2_{2}}\set{ X>2} = e^{-1}$
we obtain $\Prob\set{\norm{\eta}^2\geq 1}>e^{-1}$ and, consequently,
\begin{align*}
\Prob\set{\Norm{(v,\eta)}^{2(d-1)}> 2^{d-1}}
=\Prob\set{(1+\norm{\eta}^2)^{d-1}>2^{d-1}}
=\Prob\set{\norm{\eta}^{2}>1}
> e^{-1} .
\end{align*}
Plugging this into \cref{hallo77} shows
$
   \Prob\set{\mu(\cF_f,(v,\eta))\geq  d^{-1}\sqrt{2}^{d-3}} > e^{-1} \approx 0.368.
$
 This finishes the proof. \qed
{\bibliographystyle{plain}
\bibliography{literature}}
\end{document}